\documentclass[11 pt,reqno]{amsart}
\usepackage{amsmath,amsthm,amsfonts,amssymb,mathrsfs,bm,graphicx,stmaryrd,dsfont}
\usepackage[usenames,dvipsnames]{color}
\usepackage[colorlinks=true,linkcolor=blue]{hyperref}
\usepackage[letterpaper,hmargin=1.0in,vmargin=1.0in]{geometry}
\parskip 	\smallskipamount

\newtheorem{theorem}{Theorem}[section]
\newtheorem{lemma}[theorem]{Lemma}
\newtheorem{corollary}[theorem]{Corollary}
\newtheorem{proposition}[theorem]{Proposition}

\newtheorem{remark}[theorem]{Remark}

\newtheorem{definition}[theorem]{Definition}

\newtheorem{fact}{Fact}
\newtheorem{maintheorem}{Theorem}

\newtheorem{innercustomthm}{Theorem}
\newenvironment{customthm}[1]
  {\renewcommand\theinnercustomthm{#1}\innercustomthm}
  {\endinnercustomthm}

\def\N{\mathbb{N}}
\def\P{\mathbb{P}}
\def\Z{\mathbb{Z}}
\def\R{\mathbb{R}}

\def\E{\mathbb{E}}

\newcommand{\B}{\mathbb{B}}

	\renewcommand{\P}{\mathbb{P}}

\newcommand{\cB}{\mathcal{B}}

\newcommand{\cF}{\mathcal{F}}

\newcommand{\cL}{\mathcal{L}}
\newcommand{\cM}{\mathcal{M}}

\newcommand{\cP}{\mathcal{P}}

\newcommand{\cS}{\mathcal{S}}

\newcommand{\cU}{\mathcal{U}}

\newcommand{\bx}{\mathbf{x}}
\newcommand{\by}{\mathbf{y}}
\newcommand{\bY}{\mathbf{Y}}
\newcommand{\bX}{\mathbf{X}}

\newcommand{\ce}{\mathcal{E}}
\newcommand{\e}{\varepsilon}
\begin{document}
\title[Delocalization of Polymers]{Delocalization of Polymers in Lower Tail Large Deviation}
\author{Riddhipratim Basu}
\address{Riddhipratim Basu, Intenational Centre for Theoretical Sciences, Tata Institute of Fundamental Research, Bangalore, India}
\email{rbasu@icts.res.in}
\author{Shirshendu Ganguly}
\address{Shirshendu Ganguly, Department of Statistics, UC Berkeley, Berkeley, CA, USA}
\email{sganguly@berkeley.edu}
\author{Allan Sly}
\address{Allan Sly, Department of Mathematics, Princeton University, Princeton, NJ, USA}
\email{allansly@princeton.edu}
\date{}
\maketitle

\begin{abstract}

Directed last passage percolation models on the plane, where one studies the weight as well as the geometry of optimizing paths (called polymers) in a field of i.i.d.\ weights, are paradigm examples of models in KPZ universality class. In this article, we consider the large deviation regime, i.e., when the polymer has a much smaller (lower tail) or larger (upper tail) weight than typical. Precise asymptotics of large deviation probabilities have been obtained in a handful of the so-called exactly solvable scenarios,  including the Exponential~\cite{Jo99} and Poissonian~\cite{DZ1, sepLDP} cases. How the geometry of the optimizing paths change under such a large deviation event was considered in~\cite{DZ1} where it was shown that the paths (from $(0,0)$ to $(n,n)$, say) remain concentrated around the straight line joining the end points  in the upper tail large deviation regime, but  the corresponding question in the lower tail was left open. We establish a contrasting behaviour in the lower tail large deviation regime, showing that conditioned on the latter, in both the models, the optimizing paths are not concentrated around any deterministic curve. Our argument does not use any ingredient from integrable probability, and hence can be extended to other planar last passage percolation models under fairly mild conditions; and also to other non-integrable settings such as high dimensions.

\end{abstract}
\tableofcontents

\section{Introduction and Main results}

Last passage percolation models on the plane are paradigm examples of models believed to be in the KPZ universality class. In these models, vertices of $\Z^2$ are equipped with independent and identically distributed random weights. The weight of a path is the sum of the weights along it, and the last passage time between two points is obtained by maximizing the weight among all directed paths between them (see Section \ref{def1} for precise definitions). Although the asymptotic behaviour is believed to be universal under mild conditions on the passage time distribution, detailed understanding of these models has so far been mostly restricted to a handful of exactly solvable case where very fine information, both algebraic and geometric, is obtained using formulae from integrable probability. Although our results hold for last passage percolation with fairly general edge weight distribution; for the sake of concreteness we shall focus, for much of this article, on the interesting special case of the exactly solvable model with exponentially distributed edge weights, while deferring until later the extension to more general settings.

This case of Exponential directed last passage percolation (LPP) is very well studied, in particular because of the correspondence with Totally Asymmetric Simple Exclusion Process (TASEP) on a line. Using the understanding of invariant measures of TASEP, already in 1981, Rost \cite{Ro81} evaluated the limiting shape for this model; in particular he showed the following. Let $L_n$ denote the last passage time from $(0,0)$ to $(n,n)$, then $\E L_n/n \to 4$. Rost's results \cite{Ro81} in particular show that the limit shape for exponential last passage percolation is strictly concave; this, together with some basic concentration estimates (e.g.\ in \cite{Tal94}), imply that the maximal path (henceforth called the geodesic, or the polymer) from $(0,0)$ to $(n,n)$ is with high probability concentrated around the straight line joining the two points. More precise results were obtained later using exact determinantal formulae: Johansson \cite{Jo99} established the $n^{1/3}$ fluctuation of $L_n$ and a Tracy-Widom scaling limit; a more precise version of that also implies that the fluctuation of the geodesic around the diagonal line is of the order $n^{2/3}$ (see \cite{J00,BSS14} for more details).

Along with the typical behaviour; large deviation behaviour of $L_n$ (i.e., when the deviation of $L_n$ from $4n$ is linear in $n$) has also attracted attention. In fact Johansson \cite{Jo99} obtained large deviation rate functions for $L_n$, i.e., the precise rates of decay for probabilities that $L_n$ is either much larger or much smaller than typical (see Theorem \ref{t:ldp}). There has been a great deal of interest in the general theory of large deviations to understand the geometric consequences of conditioning on rare large deviation events. In this paper we study the geometry of the geodesic in last passage percolation when the passage time $L_n$ is conditioned to be atypical.

It is at least heuristically not too hard to see that, in the  upper tail large deviations regime, i.e., when the last passage time is conditioned to be macroscopically larger than typical; the geodesic is still localized with high probability around the diagonal. The picture in the lower tail large deviations regime is more complicated. We establish a contrasting delocalization result in this case. Our main result, Theorem \ref{t:deloc}, shows that conditioned on the last passage time being much smaller than typical, the geodesic is not localized around any deterministic curve. The different behaviour in the two tails is intimately connected to the different speeds at which large deviation occurs in the upper tail and lower tail regimes respectively, (see Section \ref{s:outline} for further elaboration along these lines).

In the context of Poissonian directed last passage percolation (henceforth to be referred to as Poissonian LPP) on the plane, another model in the KPZ universality class, which has the same qualitative behaviour as the exponential model, this question was investigated by Deuschel and Zeitouni \cite{DZ1}, who obtained explicit formulae for the rate function and showed that conditioned on the upper tail large deviation event, the geodesic is indeed localized around the diagonal. However the question about whether or not similar behaviour is observed in the lower tail large deviation regime was left open. We answer this question by showing a similar delocalization behaviour in the lower tail in this case as well (see Theorem \ref{informal1}).

Although both Exponential and Poissonian LPP models are exactly solvable, our arguments do not use integrability in any crucial way, and hence can be extended to more general settings. We prove a similar delocalization result in the lower tail large deviation result of last passage percolation on $\Z^2$ with some mild condition on the passage time distribution (see Theorem \ref{t:delocgen} for the precise conditions). We also establish an extension to higher dimensional LPP models (see Theorem \ref{t:delochd}). As far as we are aware these are the first results on the geodesic geometry in such a general setting.

Even though our results do not use integrability, the special case of Exponential LPP leads to certain simplifications which makes the basic argument more transparent. We shall therefore, to start with, restrict ourself to this case, while postponing the discussion about the other cases. We now move towards precise model definitions and statement of main result in the Exponential case.

\subsection{Model Definitions}\label{def1}
Let $\Pi=\{X_v: v \in \Z^2\}$ be a collection of i.i.d.\ Exponential random variables with rate $1$. Consider the following partial order $\preceq $ on $\Z^2$:  we say $(x_1,y_1)\preceq (x_2,y_2)$ if $x_1\leq x_2$ and $y_1\leq y_2$. For $u\preceq v$, a directed path $\gamma$ from $u$ to $v$ is defined as an up/right path starting at $u$ and ending at $v$, i.e., $\gamma:=\{u=u_0\preceq u_1 \preceq \cdots \preceq u_{k}= v\}$ is a path in $\Z^2$ where for each $i$, $u_i-u_{i-1}=(1,0)$ or $(0,1)$. For a directed path $\gamma$ as above, the passage time (or, as we shall often say, length) of $\gamma$, denoted $L(\gamma)$ is defined as $\sum_{i=0}^kX_{u_i}$.

\begin{definition}
\label{geodesic1}
For $u\preceq v\in \Z^2$, define the last passage time $L(u,v)$ from $u$ to $v$ by
$$L(u,v):=\max_{\gamma} L(\gamma)$$
where the maximum is taken over all directed paths from $u$ to $v$. The maximizing path will be called the \emph{geodesic} between $u$ and $v$ which will be denoted by $\Gamma(u,v)$\footnote{Observe that there is almost surely a unique maximizing path between any two vertices, by continuity of Exponential random variables.}.
\end{definition}

Let us now introduce some notations. For $x,y\in \Z_{+}$, we shall denote by $L_{x,y}$ the last passage time from $\mathbf{0}:=(0,0)$ to $(x,y)$. When $x=y(=n$, say) we shall simplify the notation even further and denote the last passage time by $L_n$. The geodesic from $\mathbf{0}$ to $\mathbf{n}:=(n,n)$ shall be denoted $\Gamma_{n}$.
Let $\gamma: [0,1]\to [0,1]$ denote a continuous increasing function with $\gamma(0)=0$ and $\gamma(1)=1$. We define the $(\e,n)$-cylinder neighbourhood of $\gamma$, denoted $\gamma_{n}^{\e}$ by
\begin{equation}\label{sausage100}
\gamma_{n}^{\e}=\biggl\{(x,y)\in \llbracket 0, n\rrbracket^2: |y-n\gamma(n^{-1}x)|\leq \e n\biggr\};
\end{equation}
i.e., it is a cylinder of width $\e n$ around the path from $\mathbf{0}$ to $\mathbf{n}$ that is obtained by scaling up $\gamma$. The following result is standard (see e.g.\ \cite{BSS17B} for a much stronger result).

\begin{theorem}
\label{t:diag}
Let $\mathbb{I}$ denote the identity function on $[0,1]$, and let $\e>0$ be fixed. Let $\ce_n$ denote the event that $\Gamma_n$ is contained in $\mathbb{I}_n^{\e}$. Then $\P(\ce_n)\to 1$ as $n\to \infty$.
\end{theorem}

Observe that this result asserts that the geodesic $\Gamma_n$ with high probability has Hausdorff distance $o(n)$ from the diagonal line joining $\mathbf{0}$ and $\mathbf{n}$. This is a very general result and essentially uses only the strict concavity of the limit shape. For exactly solvable models of last passage percolation, quantitatively optimal stronger variants of this result is available \cite{J00, BSS14}. See Section \ref{pwoc} for more elaboration on this.
As mentioned in the introduction, our main result in this paper shows that this behaviour changes in the lower tail large deviations regime. Formally, for $\delta\in (0,4)$ let $\cL_{\delta}$ denote the event $L_{n}\leq (4-\delta)n$. We have the following theorem.

\begin{maintheorem}
\label{t:deloc}
For each $\delta\in (0,4)$, and an increasing continuous function, $\gamma: [0,1]\to [0,1]$, with $\gamma(0)=0$ and $\gamma(1)=1$, for any $\e>0$, there exists $\e'>0,$ such that for all large enough $n,$
$$\P(\Gamma_n \subseteq \gamma_n^{\e'} \mid \cL_{\delta})\le \e.$$
\end{maintheorem}

\begin{figure}[hbt]\label{expl2}
\centering
\includegraphics[scale=.5]{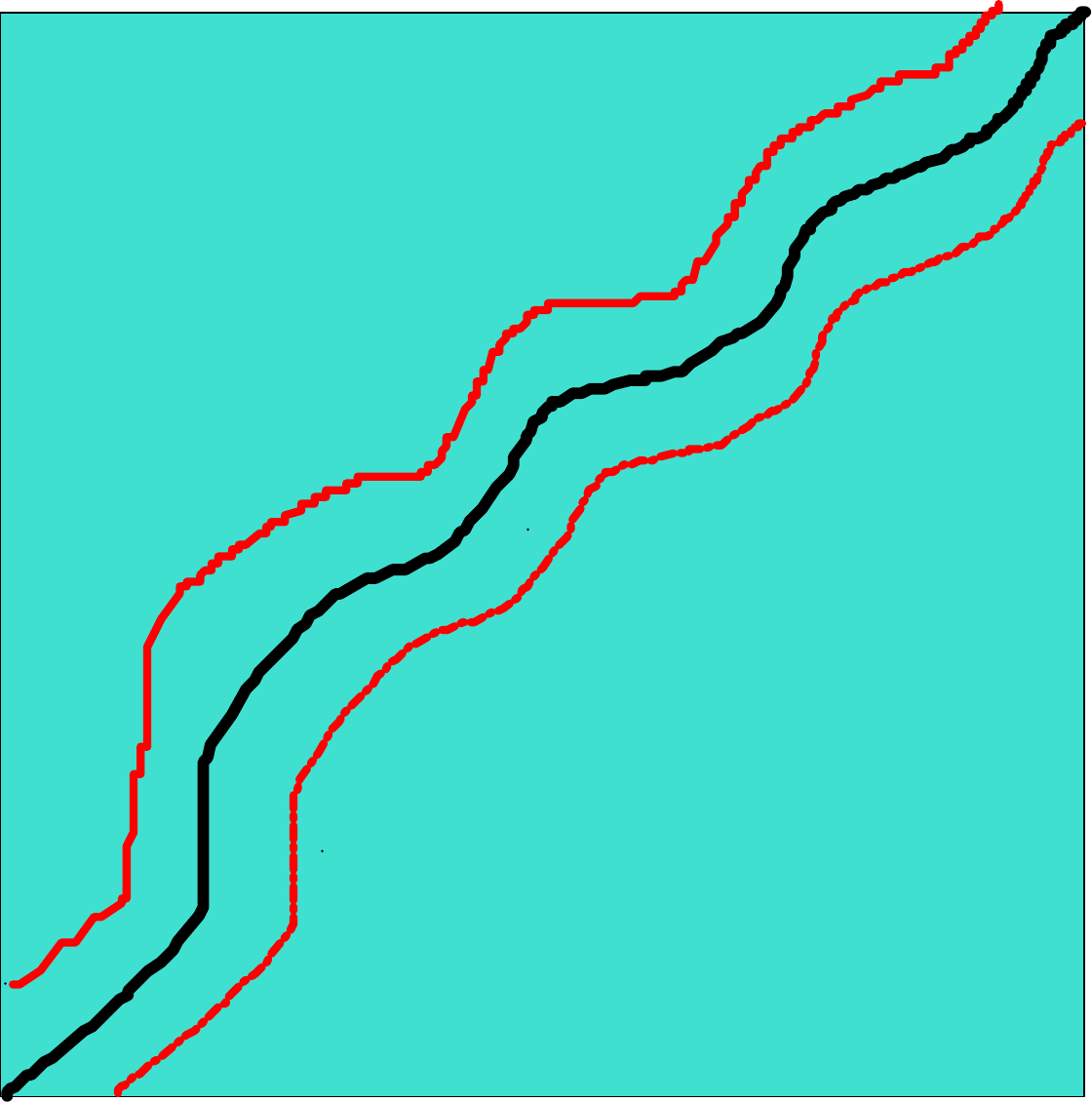}
\caption{Scale-up of a deterministic curve $\gamma$ from $(0,0)$ to $(n,n)$ and its $\e$-cylinder neighbourhood. Theorem \ref{t:deloc} says that on the lower tail large deviation event, the geodesic $\Gamma_{n}$ is unlikely to be contained in such a cylinder for any fixed curve $\gamma$.}
\end{figure}

Thus Theorem \ref{t:deloc} asserts that there does not exist any deterministic curve $\gamma,$ such that $\Gamma_n$ is localized around that curve with high probability. In the next section we continue our discussion regarding typical and rare behaviour of polymer models and variants of Theorem \ref{t:deloc}, in other related settings including a class of non-integrable models, and compare such results with existing literature.
\subsection{Background, previous works and our contributions}
\label{pwoc}

For polymer models in the KPZ universality class (last passage percolation is a general example of a zero temperature polymer model), the geometric properties of polymers (i.e., the geodesics in our context) has been an object of fundamental study. The three scaling exponents $(1,1/3,2/3)$ are characteristic of the KPZ universality class that corresponds to polymer length, length fluctuation and spatial decay of correlation. In the context of last passage percolation it can be illustrated as follows. For the geodesic from $\mathbf{0}$ to $\mathbf{n}$, the length of the geodesic is of the order $n^{1}$, the fluctuation of the length is of the order $n^{1/3}$ and the typical distance of the geodesic to the straight line joining the two points in of the order $n^{2/3}$. As mentioned before, this behaviour is expected to be universal for last passage percolation under mild condition on the passage times, but is rigorously known only for a handful of models including last passage percolation on $\Z^2$ with exponential and geometric weights, and also Poissonian LPP on  $\R^2$ (for precise definition, see Section \ref{s:poi}).

Length fluctuations of the order $n^{1/3}$ was first proved for the Poissonian LPP in the seminal work of Baik, Deift and Johansson \cite{BDJ99}, who also prove the weak convergence to GUE Tracy-Widom distribution after suitable centering and scaling. The corresponding result in Exponential last passage percolation is due to Johansson \cite{Jo99}. For completeness, let us recall the standard results in this case. Recall the last passage time $L_{x,y}$ from $(0,0)$ to $(x,y)$. The first order behaviour for $L_{x,y}$ was established in \cite{Ro81}.
\begin{theorem}
\label{t:lln}. Let $x,y>0$ be fixed real numbers. Then
$$\lim_{n\to \infty} \frac{1}{n}\E L_{\lfloor nx \rfloor, \lfloor ny \rfloor}= G(x,y)= (\sqrt{x}+\sqrt{y})^2.$$
\end{theorem}
In particular, for $x=y=1$, this shows that $\E L_n= (4+o(1))n$. Moreover, for $x,y>0$, it is known that $n^{-1/3} \left(L_{\lfloor nx \rfloor, \lfloor ny \rfloor}-nG(x,y)\right)$ converges weakly. It is however, much easier to show (and the argument is much more general) concentration of $L_{\lfloor nx \rfloor, \lfloor ny \rfloor}$ at scale $\sqrt{n}$ (see for example, \cite{Tal94}).  Also important is the boundary of the limit shape $\{(x,y)\in \R_{+}^{2}: G(x,y)=1\}$ which can be observed to be strictly concave. This implies that the growth (of $L_{x,y}$) is fastest in the diagonal direction $(1,1)$. This, together with the above concentration result establishes that $\Gamma_{n}$ is concentrated around the straight line joining $\mathbf{0}$ and $\mathbf{n}$ as already stated in Theorem \ref{t:diag} that was established in \cite{DZ2} in the context of Poissonian LPP and its variants. Sharp exponent of transversal fluctuation (i.e., the maximum vertical distance between a geodesic from $\mathbf{0}$ to $\mathbf{n}$ and the diagonal line joining two points) was obtained in \cite{J00} which showed that the maximum transversal fluctuation in $n^{2/3+o(1)}$ with high probability in the Poissonian case. The same holds in the Exponential case; see \cite[Theorem 11.1] {BSS14} for the statement of a quantitatively sharper result.

Work on large deviations in polymer models goes back to Kesten~\cite{Kes86} in 1986 who considered large deviation problems in the related setting of first passage percolation. For $\delta>0$, let $\mathcal{U}_{\delta}=\{L_n\geq (4+\delta) n\}$ denote the upper tail event analogous to  $\cL_{\delta}$ already defined in the statement of Theorem \ref{t:deloc}. A straightforward adaptation of the argument of~\cite{Kes86} to show that the log probabilities for the upper tail event scales as $n$, whereas the log probabilities scale as $n^2$ for the lower tail. In~\cite{Jo99} a precise rate function was established:

\begin{theorem}[\cite{Jo99}]
\label{t:ldp}
There exists function $I_{u}(\delta), I_{\ell} (\delta)$ such that $I_{u}(\delta)\in (0,\infty)$ for all $\delta>0$ and $I_{\ell}(\delta)\in (0,\infty)$ for $\delta\in (0,4)$ such that
\begin{align*}
\lim_{n\to \infty} \frac{1}{n}{\log(\P(L_n\ge (4+\delta)n))}&= -I_u(\delta)\\
\lim_{n\to \infty} \frac{1}{n^2}{\log(\P(L_n\le (4-\delta)n))}&= -I_\ell(\delta).
\end{align*}
\end{theorem}
Existence of the rate function for the upper tail follows from a standard sub-additive argument, and using appropriate concentration estimates one can show non-triviality of the rate function. However the proof of the lower tail result depends on the exact determinantal formulae for this specific model. Johansson provides an explicit formula for the upper tail rate function, whereas the lower tail rate function was not evaluated explicitly.

The analogue of Theorem \ref{t:ldp} was proved for Poissonian LPP in \cite{DZ1, sepLDP} using connections to longest increasing subsequences for permutations and the RSK correspondence to Young Tableaux. In~\cite{DZ1}, the geometry of geodesics in the large deviation regime has been investigated in this setting. Conditioned on the upper tail large deviation event, it was shown that the geodesics remain localized around the diagonal. Their method of proof can be adapted to our setting of the Exponential LPP in a straightforward manner to show that in the notation of Theorem \ref{t:diag}, except for an event of exponentially (in $n$) small conditional probability given $\cU_{\delta},$ we have $\Gamma_{n}\subseteq \mathbb{I}_n^{\e}$.

The analysis of the lower tail rate function in \cite{DZ1} involves solving a variational problem for shapes of Young Tableaux and did not provide any geometric information about the optimal path. In particular, whether or not the path is localized around a deterministic curve conditioned on the lower tail large deviation event was mentioned as Open Problem 2 in \cite{DZ1}. Our arguments proving Theorem \ref{t:deloc} can be adapted in this setting too to show that conditioned on the lower tail large deviation event the path is not localized around any deterministic curve.  However, a formal statement of the result in this setting needs new notation and is postponed to Section \ref{s:poi} (see Theorem \ref{t:delocLPP} there), and for the moment we only present the following informal statement.

\begin{customthm}{2}[Informal]
\label{informal1}
Fixing $\delta$, for any increasing $\gamma: [0,1]\to [0,1]$ with $\gamma(0)=0$ and $\gamma(1)=1$, there exist $\e>0$, such that
$$\P(\ce_{\gamma, n} \mid \cL_{\delta})\to 1$$
as $n\to \infty$,
where  $\ce_{\gamma, n}$ denotes the event that there exists a geodesic $\Gamma_{n}$ between $\mathbf{0}$ and $\mathbf{n}$ that is not contained in $\gamma_{n}^{\e}$ (the sausage of width $\e n$ around a properly scaled version of $\gamma$).
\end{customthm}

Note the subtle qualitative difference between Theorems \ref{t:deloc} and \ref{informal1}. In the former, by continuity of the  exponential variables, the geodesic $\Gamma_n$ is well defined, whereas for the latter, due to the discrete nature of the setting, there are several geodesics between points $\mathbf{0}$ and $\mathbf{n}.$ Hence whereas the former says that conditioned on $\cL_{\delta},$ with high probability the geodesic does not lie within a narrow sausage around any deterministic function, the latter theorem shows existence of several geodesics outside such a sausage, while not precluding existence of some geodesic within the sausage. We elaborate more on this in Sec \ref{s:poi}.
Thus Theorems~\ref{t:deloc} and~\ref{t:delocLPP} exhibit a transition from localization to delocalization in going from the upper tail to the lower tail of the large deviation regime.

Finally we emphasize that even though the above mentioned theorems are in the integrable setting of Exponential LPP or Poissonian LPP on $\Z^2$ and $\R^2$ respectively, our argument actually is far more general. In Section \ref{s:general} we consider LPP on $\Z^2$ with general passage times which makes the model lose its integrable structure. Under some smoothness conditions, e.g., monotonicity or log-concavity, on the density of the passage time distribution, we prove a delocalization analogous to Theorem \ref{t:deloc} (see Theorem \ref{t:delocgen} for a precise statement.) In Section \ref{highdim}, we establish similar phenomenon in high dimensional LPP models that are not integrable; (see Theorem \ref{t:delochd}.) We believe that our arguments should be generalizable to show delocalization in other related non-integrable settings such as first passage percolation on $\Z^d$ (the relevant large deviation event there is the upper tail event), but we do not explore the latter in this article.

While on the topic of non-integrable models, we should mention that the absence  of exact formulas present significant challenges towards just showing the existence of a large deviation rate function analogous to Theorem \ref{t:ldp}. Beyond the exactly solvable regime, such a result was so far only known in a single setting with special geometric constraint \cite{CZ}. In a forthcoming article \cite{BGS17B}, we prove the existence of the rate function in the context of general first and last passage percolation models. However, our results in this paper do not require the existence of the rate function and only relies on the fact that the speed of decay for the lower tail large deviation probability is $n^2$.

\subsection{Outline of the Proof}
\label{s:outline}

We present  the key ideas in the proof of Theorem \ref{t:deloc} in this subsection.
We start with the observation that the  speed $n^2$ for the log-probability of $\cL_{\delta}$ in Theorem \ref{t:ldp} follows by noting that given $\delta$ there exists $C>0,$ such that there are $\Theta(n)$ many disjoint translates of the strip of width $C$ around the main diagonal joining $\mathbf{0}$ and $\mathbf{n}$ such that the maximum weight increasing path in each of these strips  typically has  length at least $(4-\delta)n$. This implies that to achieve the event $\cL_{\delta}$, one needs all of those paths to have values which is smaller than typical. This forces $\Theta(n^2)$ many vertices to have weights smaller than typical and hence such an event would be exponential in $n^2$  unlikely.  This is the only aspect of  the large deviation event that our proof relies on (hence allows us to prove theorems in the non-integrable setting as well). All of the above is made precise in the proof of Proposition \ref{FKG} where we show that while
in the typical environment we have by law of large numbers $\sum_{\mathbf{0}\preceq v \preceq \mathbf{n}} X_{v}=n^2+O(n)$ with high probability, in fact
conditioned on $\cL_{\delta}$, with high probability
\begin{equation}\label{macrochange}
\sum_{\mathbf{0}\preceq v\preceq \mathbf{n}}X_v \le (1-c)n^2
\end{equation}
for some $c=c(\delta)>0.$ Using the above the proof of Theorem \ref{t:deloc} has broadly two parts:
\begin{itemize}
\item Strong concentration of the length of conditional geodesic: First we show that conditional on $\{L_{n}\leq (4-\delta)n\}$, with high probability $L_n$ is concentrated around $(4-\delta)n$ at scale $\frac{1}{n}$. (We elaborate why this is true in Section \ref{supcon12})
\item Anti-concentration of constrained paths in the conditional environment:
We shall show that, conditional on $\cL_{\delta}$, the best path that is restricted to stay within a deterministic set of vertices of size $\e n^2)$, is  concentrated around $(4-\delta)n$ at scale $\frac{1}{n}$, only with probability that decays to zero with $\e.$
\end{itemize}
Combining the above, Theorem \ref{t:deloc} follows easily. We now elaborate further on the above points.

\begin{figure}[h]
\centering
\begin{tabular}{cc}
\includegraphics[scale=.55]{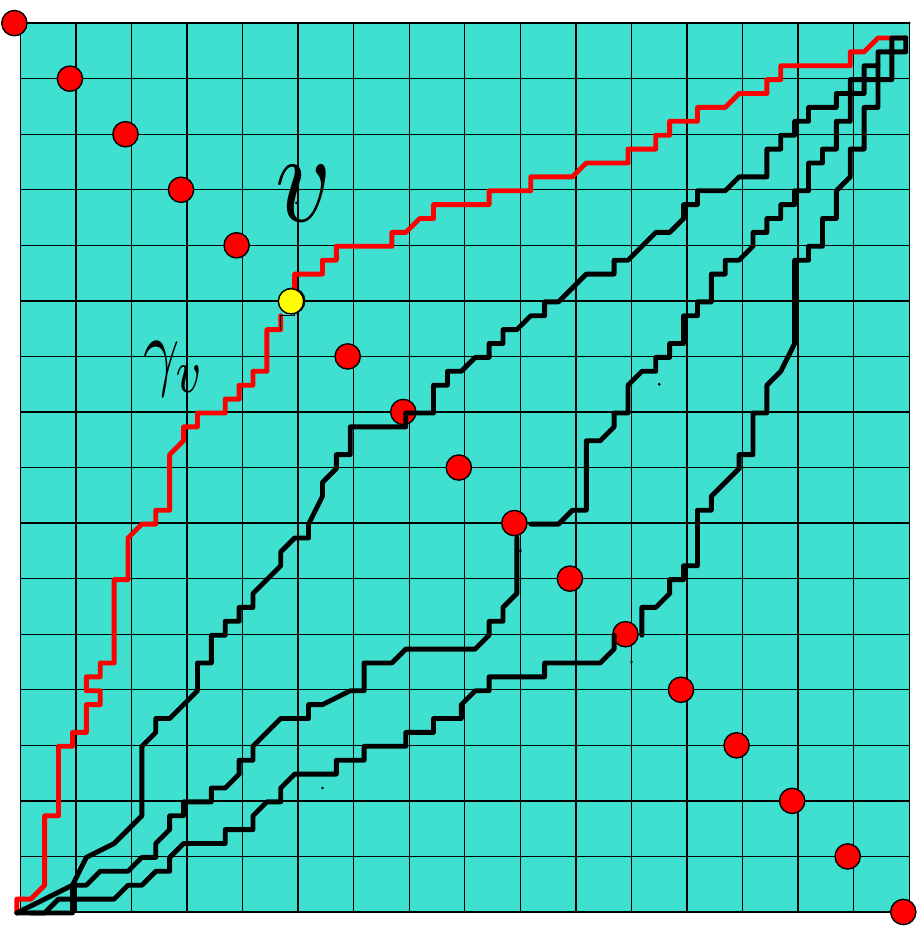} &\quad\quad\quad\quad\quad\quad\includegraphics[width=0.31\textwidth]{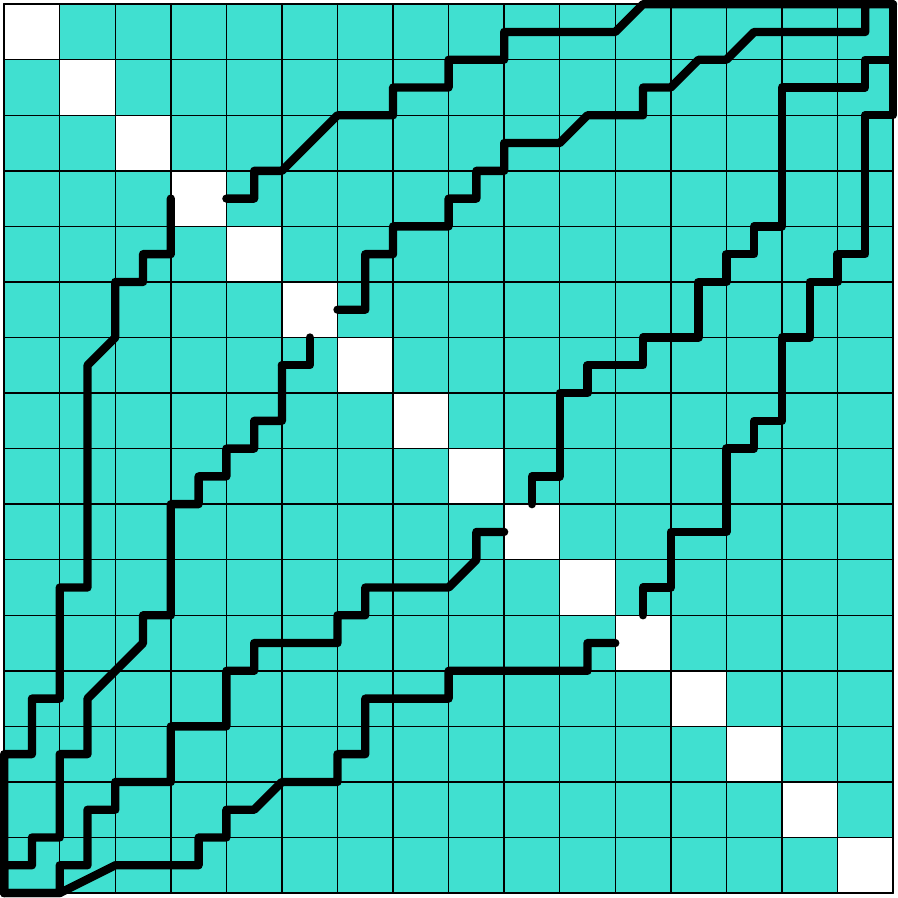} \\
(a) & \quad\quad\quad\quad\quad\quad(b)
\end{tabular}
\caption{(a) To prove super-concentration, we expose the environment except for one anti-diagonal $D$. The conditional law of $X_v$ for $v\in D,$ given the rest of the the vertices, is an Exponential random variable conditioned to be less than $M_v$ where $M_v=(4-\delta)n-|\gamma_v|$ where $\gamma_v$ is the best path passing through $v$.  Moreover $X_v$ for $v\in D$ are independent of each other. Thus $(4-\delta)n-L_n\le \inf_{v\in D} M_v-X_v$ and the latter quantity can be seen to be $O(1/n)$ if $M_v<M$ for a fixed constant $M$ for $\Theta(n)$ many $v$'s. (b) The same proof strategy works for Poissonian LPP; however in this case we expose the point process everywhere except inside small boxes along the anti-diagonal. }
\label{fig1}
\end{figure}

\subsubsection{Strong concentration of conditional geodesic}\label{supcon12}
We have the following theorem.

\begin{theorem}
\label{res1}
Fix any $\delta\in (0,4)$. Given any $\e>0$ there exists $H>0$ such that
$$\P\left(L_n\ge (4-\delta)n-\frac{H}{n}\mid \cL_\delta\right) \ge 1-\e.$$
\end{theorem}
To see why this should be true, note that for any anti-diagonal $D$ i.e., a set of the form $\{(x,y):\, x+y=k\},$ (see Figure \ref{fig1} where the anti-diagonal for $k=n$ is highlighted) the typical value of $\sum_{v\in D}X_v$ is $|D|+O(\sqrt{|D|}),$ since exponential variables have mean one and the sum is well concentrated.
However \eqref{macrochange} implies that, conditional on $\cL_{\delta}$, typically
there exists $\Theta(n)$ many anti-diagonals $D$ (each with size linear in $n$) such that $\sum_{v\in D} X_v \le (1-c)|D|$ for some $c=c(\delta)>0$.
At this point we use the following observation: fixing any anti-diagonal $D$, if we condition on the weights of the remaining vertices $\{X_v: v\notin D\}$ (the green region in Figure \ref{fig1} a.), and the event $\cL_{\delta},$ then  conditionally, $\{X_v : v \in D\}$ (the weights on $D$) is nothing but a collection of independent variables where $X_v$ follows the law of standard exponential variable conditioned to be less than $R_v$ which is a deterministic function of the weights of the vertices in the green region.
More precisely, $R_v=(4-\delta)n-M_v$ where
\begin{equation}\label{opt234}
M_v=\max_{v\in \gamma}|\gamma|-X_v,
\end{equation} where the maximum is taken over all directed paths $\gamma$ from $\mathbf{0}$ to $\mathbf{n}$ that pass through $v.$ Note that any directed path from $\mathbf{0}$ to $\mathbf{n}$ intersects any anti-diagonal  at exactly one point. Thus $M_v$ (and hence $R_{v}$) is indeed a deterministic function of all the variables $\{X_v:v\notin D\}.$ At this point we shall conclude that for anti-diagonals $D$ such that $\sum_{v\in D}X_v \le (1-c)|D|,$ with high probability $R_v$ must be uniformly bounded by some $M$ for at least $\Theta (n)$ many vertices in $D$.

It might be useful to think about this in the following way. For each $v$ on an anti-diagonal, the event $\cL_{\delta}$ forces a `barrier' $R_v$, and $X_{v}$ is an Exponential random variable conditioned not to exceed this barrier. Clearly the $v$ such that $X_v$ gets closest to its barrier $R_v$ will be the unique point on the anti-diagonal that the maximum length path passes through and consequently $(4-\delta)n-L_n$ is $R_v-X_v$. Now let us pretend that all the barriers are $M$ (this is not true but the calculation is qualitatively the same as long as $R_{v}\leq M$ for a linear number of $v\in D$), and the proof is complete by observing that $M-\max_{1\leq i\leq n} Y_{i}=O(\frac{1}{n})$ with high probability where $Y_{i}$ are i.i.d.\ $\mbox{Exp}(1)$ variables conditioned to be at most $M$.

\subsubsection{Anti-concentration of restricted paths in the conditional environment}
To make a formal statement we need the following notations.  For any $A\subset \Z^2,$  let $\Gamma_n(A)$ (resp.\ $L_n(A)$) be the longest directed path (resp.\ length) which lies entirely in $A$.   We then have the following theorem.

\begin{theorem}
\label{res2}
Fix any $\delta>0.$ Then given any $H$ and $\e>0$ there exists $\e'>0$ such that for every deterministic set $A\subseteq \llbracket 0,n \rrbracket^2$, with $|A|\leq \e' n^2$ we have
$$\P\left(L_n(A) \ge (4-\delta)n -\frac{H}{n}\mid \cL_{\delta}\right)\le \e.$$
\end{theorem}

Let us attempt to give an informal reasoning for the above theorem. Clearly $L_n(A)$ only depends on $\{X_{v}:v\in A\}$. We can decompose this data into two parts: $S=\sum_{v\in A}X_v$ and $\{\frac{X_{v}}{S}: v\in A\}$ which are independent by the properties of Exponential random variables. It is well known that $S$ is distributed as a $\mbox{Gamma}(|A|)$ random variable. Now conditioned on  $\{X_{v}/S: v\in A\}$ and the field outside $A$ and $\cL_{\delta}$ (obviously all the pieces of the data have to be compatible with $\cL_{\delta}$ for this to make sense), by similar arguments as in Sec \ref{supcon12}, the conditional distribution of $S$ is the distribution of a $\mbox{Gamma}(|A|)$ variable conditioned to be less than some value $S_0$ where $S_0$ is a measurable function of the sigma field being conditioned on and is $\Theta(n^2)$ for most of the realizations of the data. It then follows that the typical value of $S$ under the above conditioning is around $(1-\frac{1}{|A|})M.$ As argued before as well, $(4-\delta)-L_{n}(A)$ is governed by how close $S$ is to $M$; and in particular the event $L_n(A) \ge (4-\delta)n -\frac{H}{n}$ would require $S>(1-\frac{C}{n^2})M,$ for some $C$ which just depends on $\delta$ and $H$. Thus taking $A\le \e' n^2$ for some small enough $\e'$ will complete the proof.

The main work of this paper goes into proving Theorem \ref{res1} and Theorem \ref{res2}. We shall, however, first provide the immediate proof of Theorem \ref{t:deloc} using these.

\begin{proof}[Proof of Theorem \ref{t:deloc}]
Fix $\delta\in (0,4)$ and $\e>0$. Let $H>0$ be such that the conclusion of Theorem \ref{res1} holds for this choice of $\delta$ and $\e$, and let $\e'$ be such that the conclusion of Theorem \ref{res2} holds for this choice of $\delta, \e$ and $H$.Observe that the set $\gamma_{n}^{\e'}$ (see \eqref{sausage100}) has size $\e' n^2.$ Thus choosing the set $A$ in Theorem \ref{res2} to be $\gamma_n^{\e'}$, it follows that
\begin{align*}
&\P(\Gamma_n \subseteq A \mid \cL_{\delta})=\P(\Gamma_n= \Gamma_n (A) \mid \cL_{\delta})\\
&\le \P(L_n\le (4-\delta)n-\frac{H}{n}\mid \cL_{\delta})+\P(L_n(A)\ge (4-\delta)n-\frac{H}{n}\mid \cL_{\delta}) \\
& \le \e+\e.
\end{align*}\end{proof}

We remark that the choice of the definition of the $(\e,n)$
cylinder in \eqref{sausage100} is not canonical and similarly one can also take the paths $\gamma$ in the statement of Theorem \ref{t:deloc}, to be directed paths instead of graphs of honest functions. Minor variants of the above arguments would yield the same result in such situations and we omit the details for brevity.

\subsubsection{Generalization to other settings:}
\begin{enumerate}
\item
In the case of Poissonian LPP,  due to the discrete nature of the problem we have extreme concentration, i.e., with high probability the length of the geodesic equals the largest integer smaller than the barrier value of  $(2-\delta)n$ forced by the event $\cL_{\delta}.$ This is achieved using the same strategy as above, however instead of anti-diagonals of vertices, we now analyze anti-diagonals formed by small boxes (see Figure \ref{fig1} (b)).
\item
A close inspection at the above proof technique for anti-concentration, reveals that the reliance on the properties of Gamma distribution is not crucial. In fact in the proof of Theorem \ref{t:delocgen}, we prove anti-concentration under pretty general distributional assumptions where $S$ does not follow a nice distribution.  Note however that the above anti-concentration result is only expected if the distribution of $S$ has some smoothness. This is the reason why in the setting of Poissonian LPP we do not expect such a result to hold leading to a difference in the nature of Theorems \ref{t:deloc} and \ref{informal1} (see the discussion following the statement of the latter).
\end{enumerate}

\subsection{Organization of the paper}
The rest of this paper is organized as follows. In Section \ref{s:sc} we prove Theorem \ref{res1}. We prove the anti-concentration estimate Theorem \ref{res2} in Section \ref{s:ac}. The remaining sections are devoted to various extensions and generalizations. In Section \ref{s:poi}, we prove Theorem \ref{informal1} and answer an open question of \cite{DZ1}. In section \ref{s:general}, we prove delocalization for general last passage percolation models on the plane, going beyond the exactly solvable regime. We finish with discussions of higher dimensional extensions in Section \ref{highdim}.

\subsection*{Acknowledgements} The authors thank Ofer Zeitouni for useful conversations. Research of RB is partially supported by a Simons Junior Faculty Fellowship and a Ramanujan Fellowship by Govt. of India. SG is supported by a Miller Research Fellowship.

\section{Strong concentration of the conditional geodesic}
\label{s:sc}
In this section we prove Theorem \ref{res1}. The crucial first step is to obtain a geometric interpretation of the $n^2$ speed for the LDP for the lower tail, showing that under the conditioning the sum of all edge weights becomes macroscopically smaller.

\begin{proposition}
\label{FKG}
Given $\delta>0$ there exists $\e\in (0,\frac{1}{16})$ such that
$$\P(\sum_{v\in \llbracket 0,n \rrbracket ^2}X_v \le (1-4\e)n^2\mid \cL_\delta) \ge 1-e^{-cn},$$
for some $c=c(\delta)$ and for all $n$ large enough. (We write $4\e$ instead of $\e$ to avoid notational cluttering later.)
\end{proposition}

\begin{figure}[hbt]\label{fig2}
\centering
\includegraphics[scale=.5]{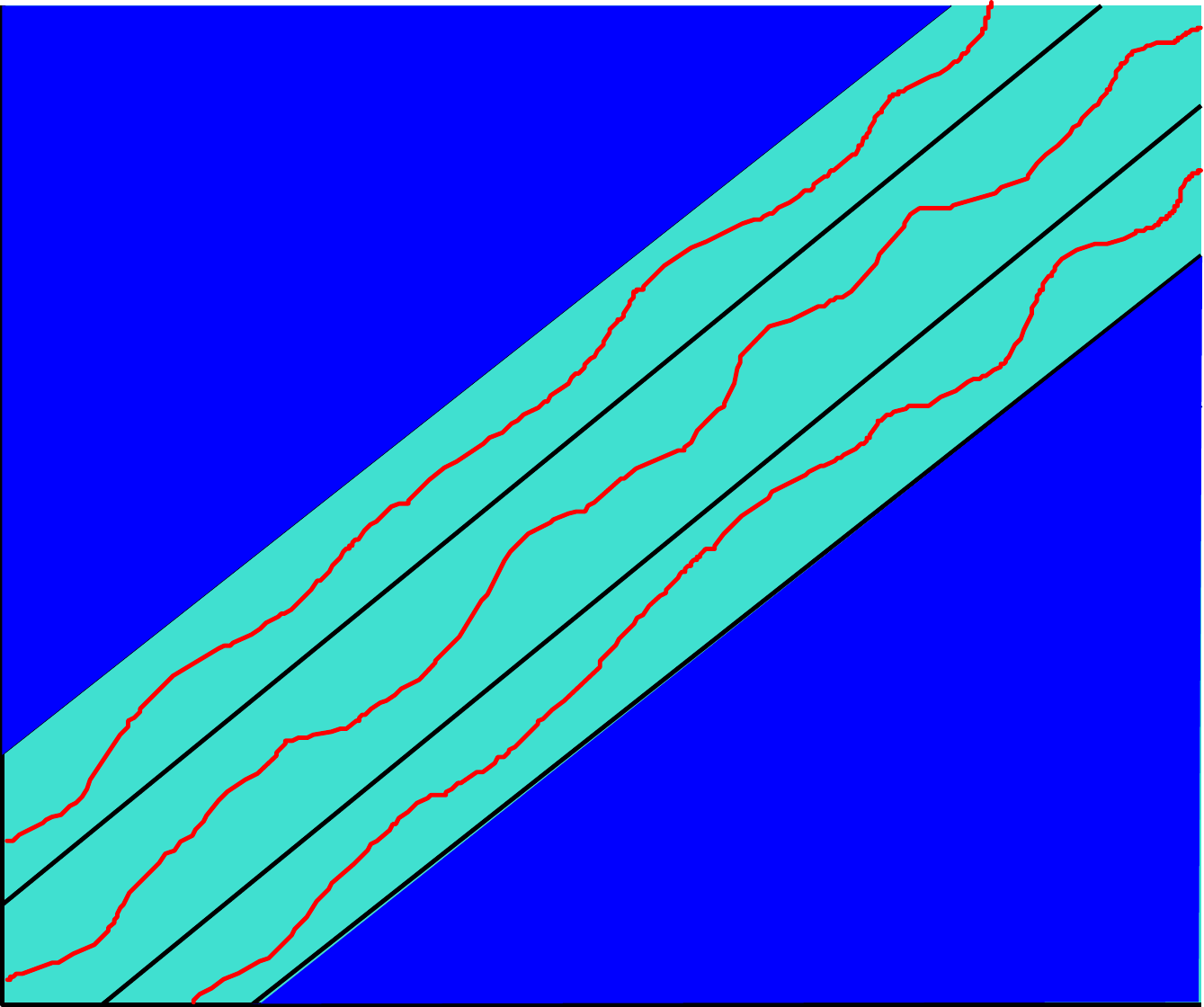}
\caption{Disjoint strips of constant width have disjoint paths of weight at least $(4-\delta/2)n$ and hence the total weight in each of these strips in the large deviation noise field must be at least $\frac{\delta}{2}n$  smaller than the typical counter part.}
\end{figure}

Proof of Proposition \ref{FKG} will depend on the FKG inequality, and the argument sketched in the beginning of Section \ref{s:outline} showing that the large deviation occurs at speed $n^2$. We first need to set up some notation. For $i\in \Z$ and $K\in \N$, let ${\rm{Strip}}^{K}_i$ be defined as follows.
$$ {\rm{Strip}}^{K}_{i}=\{(x,y)\in \llbracket 0,n \rrbracket^2: |x-y-4iK|\leq K\};$$
that is, ${\rm{Strip}}^{K}_{0}$ is the strip of width $2K$ around the main diagonal $\{x=y\}$, and the other strips are its translates by integer multiples of $4K.$ Note that the above definition ensures that ${\rm{Strip}}^{K}_{i}$ are disjoint for different values of  $i\in \Z$. The next lemma shows that  with high probability, in a typical environment, for any strip that is not too far from the main diagonal, there will be one path through each strip with length close to that of the longest path.
Recall the notation $L_n(A)$ from the statement of Theorem \ref{res2}.
We have the following lemma.
\begin{lemma}
\label{l:strip}
Fix $4>\delta>0$. There exists $K,c_0,c'>0$ depending on $\delta$ such that with probability at least $1-e^{-c'n}$ simultaneously for all $i\in \N$  with $i\leq c_0n$ we have
$$L_{n}({\rm{Strip}}^{K}_i)\ge (4-\frac{\delta}{2})n.$$
\end{lemma}

We first prove the following preparatory lemma.

\begin{lemma}
\label{l:strip0}
Fix $\delta\in (0,4)$. There exists $K$ sufficiently large and $c'>0$ so that for all $n$ sufficiently large we have
$$\P\left(L_{n}({\rm{Strip}}^{K}_0)\ge (4-\frac{\delta}{4})n\right)\geq 1-e^{-c'n}.$$
\end{lemma}

\begin{proof}
Using Theorem \ref{t:lln},  choose $K$ sufficiently large so that $\E L_{K} \geq (4-\frac{\delta}{8})K$. Assume without loss of generality that $n$ is a multiple of $K.$ Now we use super-additivity to argue that $$L_{n}({\rm{Strip}}^{K}_0)\ge L({\bf{0}},{\bf{K}})+L({\bf{K}},{\bf{2K}})+\ldots L({\bf{n-K}},{\bf{n}}),$$
where $L(\cdot,\cdot)$ was defined in Definition \ref{geodesic1}.
Now note that all the terms on the right hand side above are i.i.d.\ random variables distributed as $L_K$, and hence the proof follows from standard exponentially small probability bounds for deviation below the mean for sums of such variables.
\end{proof}

We are now ready to prove Lemma \ref{l:strip}.

\begin{proof}[Proof of Lemma \ref{l:strip}]
For $i\geq 0$, let $L_n^{i}$ denote the length of the best path from $(4iK,0)$ to $(n,n-4iK)$ that is contained in ${\rm{Strip}}^{K}_i$. Observe that by translation invariance, for all $n$ sufficiently large,
$L_{n}^{i}$ is distributed as  $L^{0}_{n-4|i|K},$ (see Figure \ref{fig2}).
Fixing $\delta,$  choose $K$ sufficiently large depending on $\delta$ so that the conclusion of Lemma \ref{l:strip0} holds. Now  choose $c_0$ sufficiently small so that $(4-\frac{\delta}{4})(1-4c_0K) \geq (4-\frac{\delta}{2})$. The proof is now completed by using Lemma \ref{l:strip0} and taking a union bound over all $i\in \N$ with $i\leq c_0n$.
\end{proof}

We state another simple lemma without proof, which is a straightforward  consequence of FKG inequality. Recall definition of $\Pi$ from Section \ref{def1}.  Let $\Pi^*=\{X^*_v\}_{v\in \Z^2}$ be $\Pi$ distributed conditionally on $\cL_\delta.$

\begin{lemma}\label{fkg203}
There is a coupling $(\Pi^*,\Pi)$ such that almost surely, $X^*_v\le X_v$ for all $v\in \Z^2$.
\end{lemma}

We are finally ready to prove Proposition \ref{FKG}.

\begin{proof}[Proof of Proposition \ref{FKG}]
Consider the coupling $(\Pi^*,\Pi)$ from the previous lemma.
For all $0\le i\le c_0n,$ (where $c_0$ appears in Lemma \ref{l:strip}) note that
\begin{align}\label{bnd45}
\sum_{v\in{\rm{Strip}}^{K}_i} (X_{v}-X^*_v)&\ge L^i_n-L^{i,*}_n\\
\nonumber
&\ge \frac{\delta}{2} n \text{ with probability at least } 1-e^{-cn},
\end{align}
where $L^{i,*}_n$ denotes the natural analogue of  $L^i_n$ corresponding to $\Pi^*$ and the last probability bound follows from Lemma \ref{l:strip0} and that by definition $L^{i,*}_n \le (4-\delta)n.$
The proof is now complete by summing the above inequality over $0\le i\le c_0n$ and noting that
\begin{align*}
\sum_{v\in \llbracket 0,n \rrbracket ^2}(X_v-X^*_v)&\ge \sum_{i=0}^{c_0 n}\left(\sum_{v\in{\rm{Strip}}^{K}_i} (X_{v}-X^*_v)\right)\\
& \ge \frac{\delta c_0 n^2}{2}\text{ with probability at least } 1-e^{-cn},
\end{align*}
where the last inequality follows from  \eqref{bnd45} and union bound over $0\le i \le c_0 n.$ The proof is now complete by choosing $4\e=\frac{\delta c_0}{4},$ and using the straightforward consequence of concentration of sum of exponential random variables, that, $\sum_{v\in \llbracket 0,n \rrbracket ^2}X_v\le (1+\frac{\delta c_0}{4})n^2$ with probability at least $1-e^{-cn^2}$ for some $c=c(\delta)>0.$
\end{proof}

\bigskip

Before continuing with the proof of Theorem \ref{res1}, we need to set up some more notation.

\noindent
\textbf{For the remainder of this section $\delta \in (0,4)$ and $\e$ will be fixed such that the conclusion of Proposition \ref{FKG} holds.}

For $i\in \llbracket 0,2n \rrbracket$, let $D_i$ denote the $i$-th anti-diagonal, i.e.,
\begin{align}\label{antid12}
D_{i}&=\{v=(x,y)\in \llbracket 0,n \rrbracket ^2: x+y=i\} \text{ and let, }\\
\label{nicint}
I&=\{i\in \N: |i-n|\leq (1-\frac{\sqrt{\e}}{4})n\}.
\end{align}
 We shall restrict our attention to those $D_i$'s with $i\in I$, in particular these $D_i$'s all have size linear in $n$.  For $i\in I$, let $\cF_i$ denote the sigma algebra generated by $\{X_v: v\notin D_i\}$. Let $\gamma_{v}$ denote the longest directed path from $\mathbf{0}$ to $\mathbf{n}$ that passes through $v$ (as already mentioned any such directed path intersects $D_i$ exactly once). Observe that $\gamma_v$ is $\cF_{i}$ measurable, and so is $L(\gamma_{v})-X_{v}$. For any $v\in D_i$, set
$$R_{v}=(4-\delta)n- (L(\gamma_{v})-X_{v}).$$
Notice that $R_v$ is $\cF_{i}$ measurable. We now have the following easy lemma.

\begin{lemma}
\label{l:distexcess}
Fix $i\in I$ and condition on $\cF_{i}$, so that $\cF_{i}$ is compatible with $\cL_{\delta}$ (this implies $R_{v}>0$ for each $v\in D_{i}$). Then conditional on $\cF_{i}$, and the event  $\cL_{\delta}$, the random variables $\{X_{v}: v\in D_i\}$ are independent and the conditional distribution of $X_{v}$ is given by that of an exponential random variable with rate one conditioned to be at most $R_{v}$ (we denote this law by ${\rm Exp}(0,R_{v})$).
\end{lemma}
\begin{proof}The proof follows by observing that the only effect of conditioning on $\cL_{\delta}$ and $\cF_i,$ is that, for any $v\in D_i,$ $X_v$ is restricted to be less than $R_v$ since otherwise the conditioning imposed by $\cL_{\delta}$ would be violated.
\end{proof}

The next lemma is an easy consequence of properties of truncated exponential variables.

\begin{lemma}
\label{l:trunc}
Let $X_{1}, X_2, \ldots, X_{m}$ be independent with $X_i\sim {\rm Exp}(0,f_i)$ where $f_i>0$ for all $i\in \llbracket m\rrbracket$. Let $M>0$ be such that $\E( {\rm Exp}(0,M))= (1-\e/4)$. Suppose $\#\{i: f_{i}\leq M\}\leq \e m/2$. Then there exists $c>0$ such that
$$\P\left(\sum_{i=1}^m X_{i} \leq (1-\e)m\right)\leq e^{-cm}.$$
\end{lemma}

\begin{proof}
The proof follows from observing that ${\rm Exp}(0,M)$ is stochastically increasing in $M$ and hence letting $A=\{i: f_{i}>M\}$
we see that $\P\left(\sum_{i=1}^m X_{i} \leq (1-\e)m\right) \le \P\left(\sum_{i\in A} Y_{i} \leq (1-\e)m\right)$ where $Y_i$ are i.i.d.\ ${\rm Exp}(0,M)$ variables. The proof now follows from
 exponential concentration of sums of i.i.d.\ variables.
\end{proof}
Let  $M$  as in the above lemma be fixed for the remainder of this section. For $i\in I$, let us define the event
 $$\cM_i:=\biggl\{\#\{v\in D_{i}: R_{v}\leq M\}\geq \e |D_i|/2\biggr\}.$$
Observe that $\cM_i$ is $\cF_{i}$ measurable. Also let us denote by $\mathcal{C}_{H}$, the event from the statement of Theorem \ref{res1}, i.e.,
$$\mathcal{C}_{H}:=\left\{L_{n}\ge (4-\delta)n-\frac{H}{n}\right\}.$$
We want to prove that $\mathcal{C}_{H}$ happens with large probability conditional on $\cL_{\delta}$. The following lemma is the first step in this direction that demonstrates the usefulness of the events $\cM_{i}$.

\begin{lemma}
\label{super}
For any $\e_1>0$, there exists $H>0$ sufficiently large so that in the above set up we have for each $i\in I$ (see
\eqref{nicint}),
$$\P(\mathcal{C}_H\mid \cL_{\delta}, \cF_i)\ge (1-\e_1)\mathbf{1}(\cM_i).$$
\end{lemma}
\begin{proof}
It follows from the definition of $I$ that on the event $\cM_{i}$, there exists a subset $S$ of $D_i$ with $|S|\geq \e^2 n$ such that $R_{v}\leq M$ for each $v\in S$. Using Lemma \ref{l:distexcess}, it follows that on $\cM_i$,
$$\P(\mathcal{C}_{H}^c\mid \cL_{\delta}, \cF_i)\leq \prod_{v\in S} \P\left({\rm Exp}(0,R_v)\leq R_{v}-\frac{H}{n}\right)\leq \left(1-\frac{cH}{n}\right)^{\e^2 n}$$
for some constant $c=c(M)>0$. Clearly, by choosing $H$ sufficiently large (depending on $c$ and $\e$ and $\e_1$) one can make the right hand side above smaller than $\e_1$. This completes the proof of the lemma.
\end{proof}
\begin{remark}\label{gen304}
Note that the above proof only relies on the following property of Exponential variables:  for each small enough $\e,$ and for any large enough $M,$ there exists $H$ such that for all large enough $n,$
 $\P({\rm Exp}(0,R)\in[R-\frac{H}{n},R])$ is  $\Theta(\frac{1}{\e^2 n}),$  uniformly for all $0\le R \le M,$ which is a simple consequence of the fact that the density of Exponential variables is uniformly away from zero on compact sets.
Also note that Lemma \ref{l:trunc}, does not rely on any special property of the Exponential distribution.
\end{remark}

Note that from Lemma \ref{super}, it follows that $$\P(\mathcal{C}_H\mid \cL_{\delta})\ge (1-\e_1)\P(\cM_i).$$
Thus the proof of  Theorem \ref{res1} is immediate if we could show the existence of $i\in I$ such that $\cM_i$ holds with probability very close to $1$.  However we are not quite able to show the latter and instead our basic strategy is to show that for many $i$'s in $I$ the event $\cM_{i}$  holds with significant probability though not quite close to one. The next few lemmas show why this suffices.

 To this end, it is useful to make the following definition. For $i\in I$, define
$$\cB_{i}=\left\{\sum_{v\in D_i} X^*_v\le (1-2\e)|D_i|\right\},$$ where $\e$ appears in the statement of Proposition \ref{FKG} and $\Pi^*$ was defined in Lemma \ref{fkg203}.
The following lemma relates $\cB_{i}$ to $\cM_{i}$.

\begin{lemma}
\label{l:bmest}
There exists a constant $c>0$ such that for each $i\in I$ we have,
$$\P(\cB_i \cap \cM_i^c\mid \cL_{\delta}) \le e^{-cn}.$$
\end{lemma}

\begin{proof}
Notice that it follows from Lemma \ref{l:distexcess} and Lemma \ref{l:trunc} that
\begin{equation}\label{rel34}
\P(\cB_i\mid \cL_{\delta}, \cF_i)\le e^{-c|D_i|}\mathbf{1}(\cM_i^c)+\mathbf{1}(\cM_i).
\end{equation}
This implies the lemma as $|D_i|$ is linear in $n$ for all $i\in I$.
\end{proof}

Lemma \ref{l:bmest} tells that we can essentially replace $\cM_i$ by $\cB_i$ in Lemma \ref{super}, even though $\cB_i$ is not $\cF_i$ measurable. Also observe that by Proposition \ref{FKG} (and definition of $I$) we know that it is very likely that at least one of the $\cB_{i}$'s hold. In fact we have something stronger.

\begin{corollary}
Given $\delta,$ let $\e>0$ be as in Proposition \ref{FKG}.
Then,
\begin{equation}
\label{e:manyb}
\P\left(\sum_{i\in I} \mathbf{1}(\cB_i)> \e n \middle| \cL_{\delta}\right)\ge 1-e^{-cn},
\end{equation}
for some $c=c(\delta)>0$ and all $n$ large enough.
\end{corollary}
\begin{proof} On the event $\{\sum_{i\in I} \mathbf{1}(\cB_i)< \e n\},$
for $n$ sufficiently large, we have
\begin{align*}
\sum_{v\in \llbracket 0,n \rrbracket ^2}X_v & \ge \sum_{i\in I: \mathbf{1}(\cB_i)=0}|D_i|-2\e n^2 \\
&\stackrel{\eqref{nicint}}{\ge} n^2- \e n^2 -2\e n^2\\
&> (1-4\e)n^2.
\end{align*}
The proof is now complete by Proposition \ref{FKG}.
\end{proof}

The strategy for completing the proof of Theorem \ref{res1} now is to check $\cB_i$'s one by one and control the conditional probability of $\mathcal{C}_{H}$ given a subset of them fails. We need the following lemma ($\delta$ as in Theorem \ref{res1} will be fixed throughout the sequel).

\begin{lemma}
\label{l:conditional} Fix small enough $\e>0$. Then for all large enough $n$ if  $\{i_1, i_2, \ldots, i_k\}\subset I$ be such that for any $1\le j\le k,$  $$\P(\cB_{i_j}\mid \cB^c_{i_1},\ldots,\cB^c_{i_{j-1}},\cL_{\delta})\ge \frac{\e}{8}.$$
Then there exists $H=H(\e)$ sufficiently large, such that for $1\le j\le k,$
$$\P(\mathcal{C}_H\mid \cB_{i_{j}}, \cB^c_{i_1},\ldots,\cB^c_{i_{j-1}},\cL_{\delta})\ge 1-\frac{\e}{2}.$$
\end{lemma}

Proof of Lemma \ref{l:conditional} is somewhat technical so we postpone it for the moment. To apply this lemma we need to demonstrate existence of subsets satisfying the hypothesis; and this is the content of the next lemma.

\begin{lemma}
\label{l:goodb}Fix small enough $\e$. Then for all large enough $n$ there exists a subset $J=\{i_1,i_2,\ldots, i_{k}\}\subseteq I$ such that we have
\begin{enumerate}
\item[(i)] $\P(\cup_{j\in J}\cB_j\mid \cL_{\delta})\ge 1-\frac{\e}{2}.$
\item[(ii)] For any $1\leq j\leq k$, we have $\P(\cB_{i_j}\mid \cB^c_{i_1},\ldots, \cB^c_{i_{j-1}},\cL_{\delta})\ge \frac{\e}{8}$.
\end{enumerate}
\end{lemma}

We shall postpone this proof too for the moment and complete the proof of Theorem \ref{res1} first.

\begin{proof}[Proof of Theorem \ref{res1}]
Fix $\e>0$ and let $H$ be sufficiently large so that the conclusion of Lemma \ref{l:conditional} holds and let $J\subseteq I$ be such that the conclusion of Lemma \ref{l:goodb} holds. Clearly,
\begin{align*}
\P(\mathcal{C}_H\mid\cL_{\delta})&\ge \sum_{j=1}^{k}\P(\mathcal{C}_{H}\cap\cB_{i_j}\cap \cB^c_{i_1},\cap \cdots,\cB^c_{i_{j-1}}\mid \cL_{\delta}),\\
&\ge \sum_{j=1}^{k}\P(\mathcal{C}_{H}\mid \cB_{i_j}\cap \cB^c_{i_1}\cap \cdots,\cB^c_{i_{j-1}} ,\cL_{\delta})\P(\cB_{i_j}\cap \cB^c_{i_1}\cap \ldots,\cB^c_{i_{j-1}}\mid \cL_{\delta}),\\
&\ge (1-\frac{\e}{2})\P(\cup_{i}\cB_i\mid \cL_{\delta}),\\
&\ge 1-\e,
\end{align*}
where in the last two inequalities we have used Lemma \ref{l:conditional} and Lemma \ref{l:goodb}. This completes the proof of the theorem.
\end{proof}

\bigskip

It remains to prove Lemma \ref{l:conditional} and Lemma \ref{l:goodb}. We start with the proof of Lemma \ref{l:goodb}.

\begin{proof}[Proof of Lemma \ref{l:goodb}]
We construct the set $J$ in the following manner. Let $\mathcal{J}$ denote the collection of all ordered sequences $J=\{j_1, j_2, \ldots j_{\ell}\}$ contained in $I$ such that,  for any $1\leq k \leq \ell$, we have $$\P(\cB_{j_k}\mid \cB^c_{j_1},\ldots, \cB^c_{j_{k-1}},\cL_{\delta})\ge \frac{\e}{8}.$$ Let $J_*=\{i_1,i_2,\ldots, i_{k}\}$ be the sequence in $\mathcal{J}$ that maximizes $\P(\cup_{j\in J} \cB_{j}\mid \cL_{\delta})$. Suppose now by way of contradiction that  $\P(\cup_{j\in J_*} \cB_{j}\mid \cL_{\delta}) <1-\frac{\e}{2}$.  The maximality assumption on $J_*$ implies that for any $i\notin J_*$ we have
\begin{align*}
\P(\cB_{i}\mid \bigcap_{j\in J^*} \cB_{j}^{c}, \cL_{\delta}) & < \frac{\e}{8} \text{ which implies ,}\\
\E\biggl[\sum_{i\in I}\mathbf{1}(\cB_i)\mid \bigcap_{j\in J^*} \cB_{j}^{c}, \cL_{\delta}\biggr]&\le \frac{\e n}{4},
\end{align*}
where the last inequality follows since $|I|\le 2n$ (see \eqref{nicint}).
Markov inequality now implies
$$\P\biggl[\sum_{i\in I}\mathbf{1}(\cB_i) \geq \e n \mid \bigcap_{j\in J^*} \cB_{j}^{c}, \cL_{\delta}\biggr]\le \frac{1}{4}.$$
Finally we conclude,
$$\P\biggl[\sum_{i\in I}\mathbf{1}(\cB_i) < \e n \mid \cL_{\delta}\biggr]\geq \frac{3}{4}\P(\bigcap_{j\in J_*} \cB_{j}^{c}\mid \cL_{\delta})> \frac{3\e}{8},$$
which contradicts \eqref{e:manyb}. This completes the proof of the lemma.
\end{proof}

Finally we move towards the proof of Lemma \ref{l:conditional}. We first need to set up some more notation.
It would be useful to have explicit notation for the sample space and its various projections. Let $\Omega=[0,\infty)^{\llbracket 0,n \rrbracket^2}$ be the product space on which the product measure of exponential vertex weights live. Fix $i_1,i_2,\ldots i_{k}$ such that the hypothesis of Lemma \ref{l:conditional} holds, i.e.,  for any $1\le j\le k$,
\begin{equation}\label{lb567}
\P(\cB_{i_j}\mid \cB^c_{i_1},\ldots,\cB^c_{i_{j-1}},\cL_{\delta})\ge \frac{\e}{8}.
\end{equation}
For $\omega\in \Omega$, let $\omega_{j}$ denote the projection of $\omega$ onto the co-ordinates $\llbracket 0,n \rrbracket^2 \setminus D_{i_j}$, i.e., the collection of weights of all the vertices except on the anti-diagonal $D_{i_{j}}$. Let $\Omega_{j}$ denote the set of all $\omega_{j}$'s. Throughout the sequel for brevity we will naturally identify subsets of $\Omega_{j}$ with their pre-image in $\Omega$ under the projection map.  Notice that, whether $\omega\in \cB_i^c$ or not, for $i\neq i_j$ is a deterministic function of $\omega_j$.
To improve transparency, we shall break the argument into a number of short lemmas.

Let  $\Upsilon_j\subset \Omega_{j}$ be the set of all $\omega_j$'s such that $\mathbf{1}(\cB^c_i(\omega_j))=1$ for $i=i_1,\ldots, i_{j-1}$.
Further, let $\cS_j\subset \Upsilon_j$ be the set of all $\omega_j$ such that
\begin{equation}\label{eq1}
\P(\cB_{i_j}\mid \omega_j, \cL_{\delta}) \le \e^3,
\end{equation}

We now have the following lemma.
\begin{lemma}
\label{l:cb}
In the above set-up, there exists $H$ sufficiently large depending only on $\e$, such that for any $\omega_j\in \Upsilon_j \setminus \cS_j$ we have
\begin{equation}
\label{e:cb1}
\P(C_H\cap \cB_{i_j}\mid \omega_j, \cL_{\delta})\ge \P(\cB_{i_j}\mid \omega_j, \cL_{\delta})(1-\e^{2}).
\end{equation}
\end{lemma}

\begin{proof}
For every $\omega_j \in \Upsilon_j \setminus \cS_j$ by definition
$$\P(\cB_{i_j}\mid \omega_j, \cL_{\delta}) \ge \e^{3}.$$
By \eqref{rel34} this implies $\mathbf{1}(\cM_{i_j}(\omega_{j}))=1$ (notice that $\mathbf{1}(\cM_{i_j})$ is also a deterministic function of $\omega_j$).
Now by Lemma \ref{super} (applied with $\e_1$ replaced by $\e^{5}$) it follows that there exists $H$ sufficiently large such that for any $\omega_j\in \Upsilon_j \setminus \cS_j$
$$\P(C_H\mid \omega_j, \cL_{\delta}) \ge 1-\e^5.$$
Thus
\begin{align*} \P(C_H\cap \cB_{i_j}\mid \omega_j, \cL_{\delta})&\ge \P(\cB_{i_j}\mid \omega_j, \cL_{\delta})- \P(C^c_H\mid \omega_j, \cL_{\delta})\\
& \ge \P(\cB_{i_j}\mid \omega_j, \cL_{\delta})(1-\frac{\P(C^c_H\mid \omega_j, \cL_{\delta})}{\P(\cB_{i_j}\mid \omega_j, \cL_{\delta})})\\
& \ge \P(\cB_{i_j}\mid \omega_j, \cL_{\delta})(1-\frac{\e^5}{\e^3})\\
&=\P(\cB_{i_j}\mid \omega_j, \cL_{\delta})(1-\e^2).
\end{align*}
\end{proof}

Let $H$ be fixed for the remainder of this section such that the conclusion of Lemma \ref{l:cb} holds. We also have the following corollary of Lemma \ref{l:cb}.

\begin{corollary}
\label{l:cb2}
In the set-up of Lemma \ref{l:cb}, we have
\begin{equation}
\label{e:cb2}
\P(C_H\cap \cB_{i_j}\cap (\Upsilon_j\setminus \cS_j)\mid \cL_{\delta}) \ge (1-\e^{2})\P(\cB_{i_j}\cap (\Upsilon_j\setminus \cS_j)\mid  \cL_{\delta})
\end{equation}
\end{corollary}

\begin{proof}
The proof follows by integrating both sides of \eqref{e:cb1} over $\omega_j\in \Upsilon_j \setminus \cS_j$ (with respect to the conditional density given $\cL_{\delta}$).
\end{proof}

The final piece of the proof of Lemma \ref{l:conditional} is provided by the next lemma.

\begin{lemma}
\label{l:cb3}
In the set-up of Lemma \ref{l:cb} we have,
\begin{equation}
\label{e:cb3}
\P(\cB_{i_j}\cap (\Upsilon_j\setminus \cS_j)\mid \cL_{\delta})\ge (1-\e^2)\P(\cB_{i_j}\cap \Upsilon_j\mid \cL_{\delta}).
\end{equation}
\end{lemma}
\begin{proof}
By definition of $\Upsilon_{j}$, and \eqref{lb567}, we have
$$\P(\cB_{i_j}\cap \Upsilon_j\mid \cL_{\delta})\ge \frac{\e}{8}\P(\Upsilon_j\mid \cL_{\delta}).$$
Using this and \eqref{eq1} we have
\begin{align*}
\P(\cB_{i_j}\cap (\Upsilon_j\setminus \cS_j)\mid \cL_{\delta})&= \P(\cB_{i_j}\cap \Upsilon_j\mid \cL_{\delta})-\P(\cB_{i_j}\cap \cS_j\mid \cL_{\delta})\\
&\ge \P(\cB_{i_j}\cap \Upsilon_j\mid \cL_{\delta})-\int_{\omega_j\in \cS_j}\P(\cB_{i_j}\mid \omega_j,\cL_{\delta}) {\rm d}\P(\omega_j\mid \cL_{\delta})\\
&\stackrel{\eqref{eq1}}{\ge} \P(\cB_{i_j}\cap \Upsilon_j\mid \cL_{\delta})-\int_{\omega_j\in \cS_j}\e^{3}{\rm d}\P(\omega_j\mid \cL_{\delta})\\
&\ge \P(\cB_{i_j}\cap \Upsilon_j\mid \cL_{\delta})-\e^{2}\P(\cS_j\mid \cL_{\delta})\\
&\ge \P(\cB_{i_j}\cap \Upsilon_j\mid \cL_{\delta})-\e^{3}\P(\Upsilon_j\mid \cL_{\delta})\\
&\ge (1-8\e^2)\P(\cB_{i_j}\cap \Upsilon_j \mid \cL_{\delta}),
\end{align*}
completing the proof of the lemma.
\end{proof}

Proof of Lemma \ref{l:conditional} is now immediate.

\begin{proof}[Proof of Lemma \ref{l:conditional}]
From \eqref{e:cb2} and \eqref{e:cb3} we deduce that
\begin{align*}
\P(C_H\cap \cB_{i_j}\cap \Upsilon_j\mid \cL_{\delta}) &\ge (1-\e^{2})(1-8\e^{2})\P(\cB_{i_j}\cap \Upsilon_j\mid  \cL_{\delta}).
\end{align*}
Thus for all small enough $\e$ we get,
$$\P(C_H\mid \cB_{i_j}, \Upsilon_j, \cL_{\delta})\ge 1-\e/2,$$
as required.
\end{proof}

\section{Anti-Concentration for restricted paths in conditional environment}
\label{s:ac}

Our objective in this section is to prove Theorem \ref{res2}. Recall that $\delta\in (0,4)$ is fixed as before. Let us fix $H>0$ and $\e_2>0$. Also fix $A\subseteq \llbracket 0,n \rrbracket ^2$ which is connected and contains both $\mathbf{0}$ and $\mathbf{n}$ with $|A|\leq \e' n^2$.

As mentioned before it is convenient to use a standard decoupling property of a collection of i.i.d.\ exponential variables. Recall the following well-known fact.

\begin{fact}\label{fact1}
Let $X_1, X_2, \ldots, X_{m}$ be i.i.d.\ ${\rm Exp}(1)$ variables. Let $Z=\sum_{i=1}^m X_i$ and $Y_i= \frac{X_i}{Z}$. Then $Z\sim{\rm Gamma}(m)$ with density proportional to $e^{-t}t^{m-1}$ on $[0,\infty)$ and is independent of the random vector $(Y_1,\ldots, Y_m)$ having distribution $F_{m}$. As a matter of fact, the distribution $F_{m}$ is also explicitly known (Dirichlet distribution which is the uniform distribution on the unit positive simplex $\mathcal{S}_m$ of dimension $m-1,$ i.e., $$\mathcal{S}_m=\{(y_1,y_2,\ldots,y_m): y_i\ge 0\,\,\, \forall\,\, 1\le i\le m, \sum_{i=1}^m y_i=1\}.$$ However the latter fact  will not be important for us.
\end{fact}
The above representation allows us sample the distribution of $\{X_{v}:v\in A\}$ conditional on $\cL_{\delta}$ in the following way:
We think of the field $\mathbf{X}=(X_v: v \in \llbracket 0,n \rrbracket^2)=:(\mathbf{X}_{A^c},\mathbf{X}_{A})$ as the triple $(\mathbf{X}_{A^c}, \mathbf{Y}_{A}, Z_{A})$
where $\mathbf{X}_{A^c}:=(X_v:v \in A^c),$ and similarly $\mathbf{X}_{A}:=(X_v:v \in A),$ $Z_{A}:=\sum_{v \in A} X_v,$  and $\mathbf{Y}_A:=(Y_v: v \in A)$ where $Y_v=\frac{X_v}{Z_A}.$
Thus $\mathbf{Y}_{A}$ is an element of $\mathcal{S}_{|A|}$ and $Z_{A}\sim {\rm Gamma}(|A|)$ and by the above fact, all the elements of the triple $(\mathbf{X}_{A^c}, \mathbf{Y}_{A}, Z_{A})$ are independent of each other.

\begin{lemma}
\label{l:sample}For all $(\mathbf{x}_{A^c}, \mathbf{y}_{A})$ (compatible with $\cL_{\delta}$), i.e.,  there exists some $z_{A}>0,$ with $(\mathbf{x}_{A^c}, \mathbf{y}_{A},z_A)\in \cL_{\delta},$ we have the following:
conditional on $\mathbf{Y}_{A}=\mathbf{y}_{A},\mathbf{X}_{A^c}=\mathbf{x}_{A^c}$ and the event $\cL_{\delta},$ the random variable $Z_A$ is distributed as a ${\rm Gamma}(|A|)$ variable conditioned to be in $[0,\theta_{\max}]$ for some constant $\theta_{\max}=\theta_{\max}(\mathbf{X}_{A^c},\mathbf{Y}_{A})$. Moreover
$\theta_{\max}\le \frac{(4-\delta)n}{L_n(A;\mathbf{Y}_A)},$ where $L_n(A;\mathbf{Y}_A)$ denotes the weight of the maximum weight path from $\mathbf{0}$ to $\mathbf{n}$ restricted to lie inside $A$ and passing through the environment $\mathbf{Y}_{A}.$
\end{lemma}

\begin{proof}
 The proof follows by the observation that fixing $\mathbf{Y}_{A}=\mathbf{y}_{A},\mathbf{X}_{A^c}=\mathbf{x}_{A^c},$ $L_n$ is a non-decreasing continuous function of $Z_A$ and hence the event $\cL_{\delta}$ forces $Z_{A}$ to be less than some constant $\theta_{\max}.$
Since choosing $Z_A=\frac{(4-\delta)n}{L_n(A;\mathbf{Y}_A)}$ forces $L_{n}\ge L_{n}(A)=(4-\delta)n$, the stated upper bound on $\theta_{\max}$ follows.
\end{proof}
We are now ready to prove Theorem \ref{res2}.

\begin{proof}[Proof of Theorem \ref{res2}]
Let $s=\frac{(4-\delta)n-\frac{H}{n}}{(4-\delta)n}=1-\frac{H}{(4-\delta)n^2}$. Now by Lemma \ref{l:sample}, it follows that
\begin{equation}
\label{e:resA}
\P\left(L_n(A) \ge (4-\delta)n -\frac{H}{n}\mid \cL_{\delta}, \mathbf{Y}_{A},\mathbf{X}_{A^c} \right)\le  \E \biggl[ \P(Z_{A}\geq s\theta_{\max} \mid \cL_{\delta}, \mathbf{Y}_{A},\mathbf{X}_{A^c} ) \biggr]
\end{equation}
where the expectation in the right hand side above is over the distribution of $\mathbf{Y}_{A},\mathbf{X}_{A^c}$ conditional on $\cL_{\delta}$.
To see this notice that $L_{n}(A)=L_n(A;\mathbf{Y}_{A})Z_A$ and hence if
$Z_{A}< s\theta_{\max},$ then $$L_{n}(A)< L_n(A;\mathbf{Y}_{A})s\theta_{\max} \le  L_n(A;\mathbf{Y}_{A})s\frac{(4-\delta)n}{L_n(A;\mathbf{Y}_A)}<  (4-\delta)n -\frac{H}{n}.$$
Using Lemma \ref{l:sample} and the density of Gamma distribution it follows that
\begin{align*}
\P(Z_A \ge s\theta_{\max}\mid \{Y_v\}) &=\dfrac{\int_{s{\theta_{\max}}}^{\theta_{\max}}e^{-t}t^{|A|-1}~{\rm d}t }{\int_{0}^{\theta_{\max}}e^{-t}t^{|A|-1}~{\rm d}t }\le \dfrac{\int_{s{\theta_{\max}}}^{\theta_{\max}}e^{-t}t^{|A|-1}~{\rm d}t }{\int_{0}^{s\theta_{\max}}e^{-t}t^{|A|-1}~{\rm d}t } .
\end{align*}
Doing the change of variable $t\mapsto \frac{t}{\theta_{\max}}$ we get that the above is upper bounded by
\begin{align*}
\frac{e^{-s\theta_{\max}}\int_{s}^{1} t^{|A|-1}~{\rm d}t}{e^{-s\theta_{\max}}\int_{0}^{s} t^{|A|-1}~{\rm d}t} &\leq \frac{(1-s^{|A|})}{s^{|A|}} \leq \biggl(e^{\frac{2H\e'}{4-\delta}}-1\biggr)=O(\e').
\end{align*}
where the final inequality follows by taking $n$ large enough, substituting the value of $s$ and using the inequality $(1+x) \le e^x$.
\end{proof}

\section{Delocalization in Poissonian LPP}
\label{s:poi}
In this section, we consider the question of delocalization in the setting of Poissonian last passage percolation on the plane and prove Theorem \ref{informal1}.
We begin by giving a short description of the standard set up of Poissonian directed last passage percolation on the plane analogous to the description in Section \ref{def1}. To draw obvious comparisons with the latter, we stick to the same notation. In the remainder of this section, we will only be discussing the Poissonian case and hence there should not be any scope of confusion.

Let $\Pi$ be a homogeneous rate one Poisson Point Process (PPP) on the plane. A partial order on $\R^2$ is given by $(x_1,y_1)\preceq (x_2,y_2)$ if $x_1\leq x_2$ and $y_1\leq y_2$.
For $u\preceq v$, a directed path $\gamma$ from $u$ to $v$ is a piecewise linear path that joins points $u=u_{0}\preceq u_1 \preceq \cdots \preceq u_{k}\preceq u_{k+1}=v$ where each point $u_i$
is a point of $\Pi$. Define the length of $\gamma$, denoted $|\gamma|$, to be the number of $\Pi$-points on $\gamma$.

\begin{definition}\label{geodesicLPP}
Define the last passage time from $u$ to $v$, denoted by $L(u,v)$, to be the maximum  of $|\gamma|$ as $\gamma$ varies over all directed paths from $u$ to $v$.
\end{definition}

Observe that conditional on the number of points $N(u,v)$ in the rectangle with bottom left corner at $u$ and top right corner at $v$, (these points are i.i.d.\ uniformly distributed in this rectangle), $L(u,v)$ has the same distribution of the longest increasing subsequence of a uniform random permutation of length $N(u,v)$. This model has therefore attracted attention for a long time as a Poissonized version of the classical Ulam's problem of longest increasing subsequences.

Using the same notations as in the case of Exponential last passage percolation for $x,y>0$ let $L_{x,y}$ denote the length of a longest increasing path from $\mathbf{0}$ to $(x,y)$, and $L_n$ shall denote the length of a longest increasing path from $\mathbf{0}$ to $\mathbf{n}$. The first order behaviour of $L_{x,y}$ was established in \cite{VerKer77,LogShep77} using analysis of Young Tableaux (see also \cite{AD95}).
\begin{theorem}
\label{t:llnLPP}
Let $x,y>0$ be fixed real numbers. Then
$$\lim_{n\to \infty} \frac{1}{n}\E L_{\lfloor nx \rfloor, \lfloor ny \rfloor}= G(x,y)=2 \sqrt{xy}.$$
\end{theorem}

In particular, for $x=y=1$, this implies $\E L_{n}\sim 2n$. As before, the limit shape is strictly concave and one can prove that $L_n$ is actually concentrated around $2n$. As a matter of fact, as alluded to in the introduction, this is one of the exactly solvable model, and much more is known following \cite{BDJ99}, where $n^{1/3}$ fluctuations and Tracy-Widom scaling limit was established.

For any $u,v\in \R^2$, let us consider paths that attain the length $L(u,v)$. We shall refer to such paths as longest increasing paths or geodesics as before. One difference from the Exponential LPP case is that since $L_n$ can take only a discrete set of values, typically there will be many paths $\gamma$ attaining the maximum length $L_n$.

As before for an increasing continuous function $\gamma: [0,1]\to [0,1]$ with $\gamma(0)=0$ and $\gamma(1)=1$, we define the continuous $(\e,n)$-cylinder neighbourhood of $\gamma$, as
$$\gamma_{n}^{\e}:=\biggl\{(x,y)\in [0, n]^2: |y-n\gamma(n^{-1}x)|\leq \e n\biggr\}.$$

The following analogue of Theorem \ref{t:diag} was established in \cite{DZ1}.

\begin{theorem}
\label{t:diagLPP}
Let $\mathbb{I}$ denote the identity function on $[0,1]$, and let $\e>0$ be fixed. Let $\ce_n$ denote the event that all the geodesics between $\mathbf{0}$ and $\mathbf{n}$ are contained in  $\mathbb{I}_n^{\e}$. Then $\P(\ce_n)\to 1$ as $n\to \infty$.
\end{theorem}

This model also exhibits the universal transversal fluctuation exponent of $2/3$, but this will not be important for us.

Deuschel and Zeitouni \cite{DZ1} considered the large deviation events $\cU_{\delta}:=\{L_n\geq (2+\delta)n\}$ for $\delta>0$ and $\cL_{\delta}:=\{L_{n}\leq (2-\delta)n\}$ for $\delta\in (0,2)$, and explicitly evaluated large deviation rate functions, establishing the analogue of Theorem \ref{t:ldp}.

\begin{theorem}
\label{t:ldpPoi}
There exists function $I_{u}(\delta), I_{\ell} (\delta)$ such that $I_{u}(\delta)\in (0,\infty)$ for all $\delta>0$ and $I_{\ell}(\delta)\in (0,\infty)$ for $\delta\in (0,2)$ such that
\begin{align*}
\lim_{n\to \infty} \frac{1}{n}{\log(\P(L_n\ge (2+\delta)n))}&= -I_u(\delta)\\
\lim_{n\to \infty} \frac{1}{n^2}{\log(\P(L_n\le (2-\delta)n))}&= -I_\ell(\delta).
\end{align*}
\end{theorem}

Solving variational problems about Young tableaux, \cite{DZ1} evaluated the rate functions explicitly. However, the rate function does not provide any immediate information about the geometry of the geodesics. Conditional on $\cU_{\delta}$, they showed that the geodesics are still localized around the diagonal, but they left open the question whether or not the geodesics are localized around any deterministic curve conditioned on the lower tail large deviation event $\cL_{\delta}$. The following formal restatement of Theorem \ref{informal1} answers this question negatively.

\begin{maintheorem}[Formal]
\label{t:delocLPP}
For each $\delta\in (0,2)$, there exist $\e_0>0$, such that the following holds. Fix an increasing, continuous, surjective function $\gamma: [0,1]\to [0,1]$. Let $\ce_{\gamma, n}$ denote the event that there exists a geodesic $\Gamma_{n},$ between $\mathbf{0}$ and $\mathbf{n}$ that is not contained in $\gamma_{n}^{\e_0}$. For each such $\gamma$
$$\P(\ce_{\gamma, n} \mid \cL_{\delta})\to 1$$
as $n\to \infty$.
\end{maintheorem}

As already mentioned in Section \ref{pwoc}, observe that the statement is subtly different from Theorem \ref{t:deloc}. In the previous case the geodesic was unique and we asserted that with large probability the geodesic is not localized around any deterministic curve $\gamma$; whereas Theorem \ref{t:delocLPP} only asserts that not all the geodesics can be contained in the small cylinder around $\gamma$. This is an effect of the discrete nature of the problem. As mentioned before, we shall show that even conditional on the large deviation event typically there will be many geodesics,  that attain the length $\lfloor(2-\delta)n \rfloor$, and hence there are some which are away from the deterministic curve $\gamma$ with high probability.

The proof strategy is similar to the proof of Theorem \ref{t:deloc}, however certain steps requires extra arguments whereas certain others are less complicated. The first step as in the Exponential case is to prove that in the lower tail large deviations regime, the law of large numbers for the total number of points in $[0,n]^2$ changes. For $A\subseteq \R^2$, we shall denote by $N(A)$ the number of $\Pi$-points in $A$. The following proposition is the analogue of Proposition \ref{FKG} in the Poissonian setting.

\begin{proposition}
\label{FKGPoi}
Given $\delta>0$ there exists $\e\in (0,\frac{1}{16})$ such that
$$\P(N([0,n]^2)\le (1-4\e)n^2\mid \cL_\delta) \ge  1-e^{-cn}$$
for some $c=c(\delta)$, and all $n$ large enough.
\end{proposition}

Proof of this proposition is almost identical to the proof of Proposition \ref{FKG} using the FKG inequality and we omit the details. For the remainder of this section $\delta\in (0,2)$ will be fixed and $\e$ will remain fixed so that the conclusion of Proposition \ref{FKGPoi} holds.

\subsection{Discretization}\label{disc20}
To use the strategy of the proof of Theorem \ref{t:deloc}, we shall need to discretize  the space $[0,n]^2$ into a lattice of boxes. Let $\e_{D}$ be a sufficiently small level of discretization (to be chosen later). Without loss of generality we shall assume that $\frac{n}{\e_{D}}$ is an integer, and set $m=\frac{n}{\e_{D}}-1$. Now divide the square $[0,n]^2$ into squares (disjoint except at the boundary) of side length $\e_{D}$ by lines parallel to the co-ordinate axes (see Figure \ref{fig1} (b)). Observe that, these boxes can in a natural way be identified to the vertices in $\llbracket 0,m \rrbracket^2 \subseteq \Z^2$, and we shall use this identification. For $v\in \llbracket 0,m \rrbracket^2$, let $\B_{v}$ denote the square corresponding to $v$. By $\Pi_{v}$ we shall denote the point process $\Pi$ restricted to $v$; these are independent collections of rate one Poisson processes on $\e_{D}\times \e_{D}$ boxes (we shall work on the probability one set on which there are no $\Pi$-points on the boundaries of any  $\B_v$).

With this discretization we can now directly borrow notations from Section \ref{s:sc}. For $i\in \llbracket 0,2m\rrbracket$, let $D_i$ denote the $i$-th anti-diagonal as before.
Let
\begin{equation}\label{niceint10}
I=\{i\in \N: |i-m|\leq (1-\frac{\sqrt{\e}}{4})m\}.
\end{equation}
As in Section \ref{s:sc} we shall restrict our attention to those $D_i$'s with $i\in I$, in particular these $D_i$'s all have size linear in $m$; which would imply tat $\cup_{v\in D_i} \B_{v}$ has area linear in $n\e_{D}$.
We will now follow the same strategy as before, to expose the point process, except inside the squares on an anti-diagonal $D_i$ (see Figure \ref{fig1} (a) and (b)).

Fix $i\in I$, let $\cF_i$ denote the sigma algebra generated by $\{\Pi_v: v\notin D_i\}$. Let $\gamma_{v}$ denote the best path from $\mathbf{0}$ to $\mathbf{n}$ that passes through $\B_v$. Notice here the difference from the Exponential LPP model. Previously, conditional on $\cF_i$, the best path through $v\in D_i$ was deterministic, and the only uncertainty in its length came from the weight of the vertex $v$. However here the path through $B_v$ can  enter and exit $\B_v$ at different points and hence we are forced to make somewhat more complicated definitions.

Given any box $\B_v$ with $v\in D_i$, let $\partial_{\rm{SW}}\B_v$ and $\partial_{\rm{NE}}\B_v$  denote the south-west and north-east boundaries (see Figure \ref{po123}) of the box $\B_v$. Condition on $\cF_{i}$ so that $\cF_{i}$ is compatible with $\cL_{\delta}$. Now for any $v \in D_i,$ consider the $\cF_{i}$ measurable function
$$f_{v}:\partial_{\rm{SW}}\B_v \times \partial_{\rm{NE}}\B_v\to \Z_{\ge 0}$$
given by
\begin{equation}\label{defbdry1}
f_{v}(x,y)=\lfloor(2-\delta)n \rfloor- (L(\mathbf{0},x)+(L(y,\mathbf{n})).
\end{equation}
\begin{figure}[h]
\centering
\includegraphics[scale=.5]{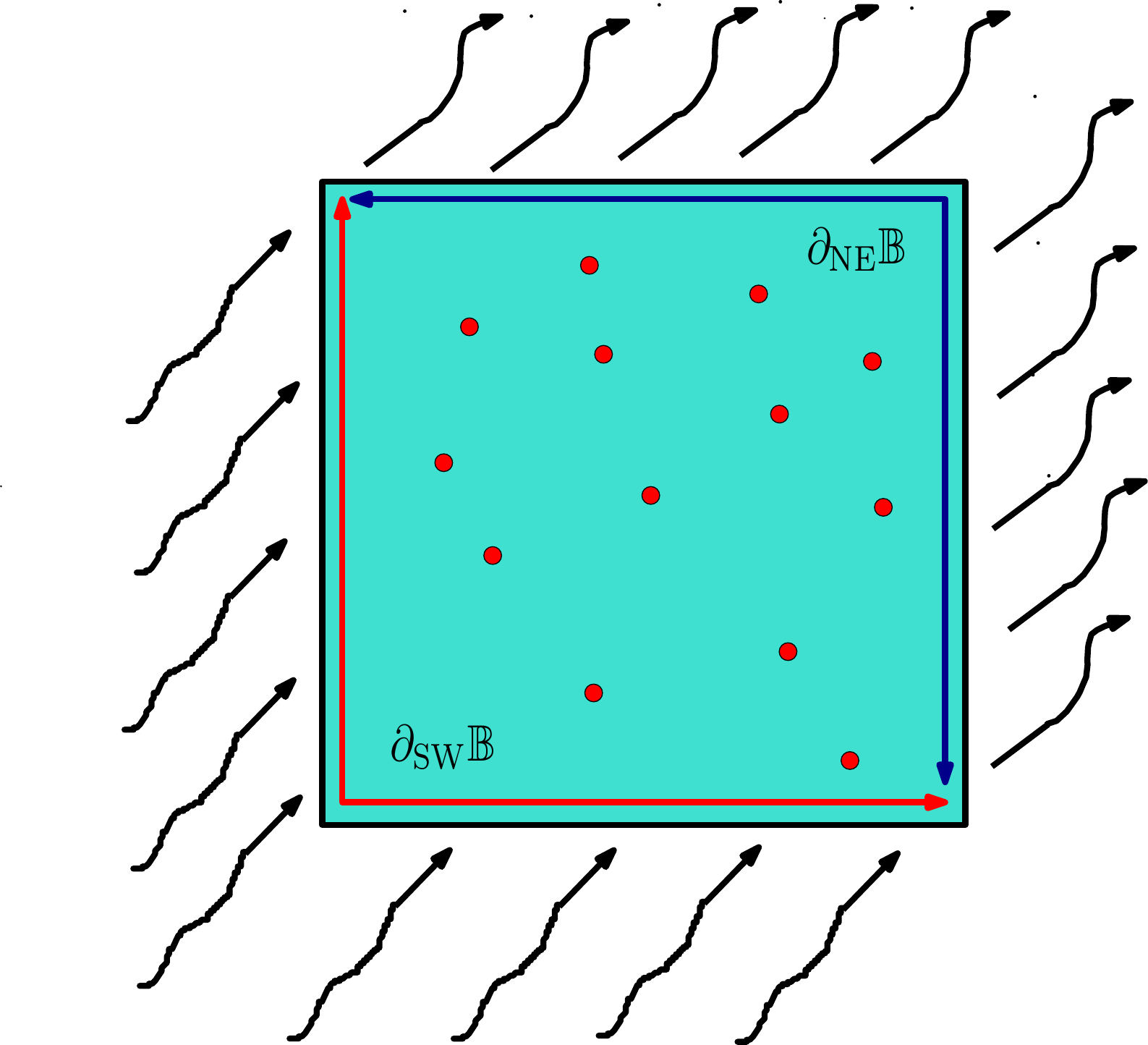}
\caption{The boundary condition imposed on the Poisson process in a certain box by exposing its complement.}
\label{po123}
\end{figure}

Note that $(L(0,x)+(L(y,n))$ is $\cF_{i}$ measurable, and obviously $f_{v}$ has to be nonnegative because other wise the data in $\cF_{i}$ will not be compatible with the event $\cL_{\delta}$. Conditional on $\cL_{\delta}$ and $\cF_{i}$ clearly, the distribution of $\Pi_{v}$ must be supported on configurations so that $L(x,y)\leq f_v(x,y)$ for all $x\in \partial_{\rm{SW}}\B_v$ and for all $y\in \partial_{\rm{NE}}\B_v$.  The result analogous to Lemma \ref{l:distexcess} in this setting is the following.

\begin{lemma}
\label{l:distexcessLPP}
Fix $i\in I$ and condition on $\cF_{i}$, so that $\cF_{i}$ is compatible with $\cL_{\delta}$ (this implies $f_{v}(\cdot,\cdot)\ge 0$ for each $v\in D_{i}$). Then conditional on $\cF_{i}, \cL_{\delta}$, the point processes $\{\Pi_{v}: v\in D_i\}$ are independent with the conditional distribution of $\Pi_{v}$ is given by that of a rate one Poisson point process under the following conditioning $$L(x,y)\le f_{v}(x,y) \,\,\,\,\forall (x,y) \in \partial_{\rm{SW}}\B_v \times \partial_{\rm{NE}}\B_v$$ (we denote this law by ${\rm{PP}}_{f_v}$).
\end{lemma}
Note that almost surely for any $v\in D_{i},$ the function $f_v(\cdot,\cdot)$ is a step function with only finitely many steps since almost surely in \eqref{defbdry1} the functions $L(\mathbf{0},x)$ and $L(y,\mathbf{n})$ are themselves piecewise constant with finitely many such pieces.
This lemma is proved in the same way Lemma \ref{l:distexcess} is proved, and hence we omit the proof.

We will also need the following notation: for any $v\in D_i$ let  $$R_v=\min_{x,y}f_{v}(x,y)$$
where the minimum is taken over $x\in \partial_{\rm{SW}}\B_v, y\in \partial_{\rm{NE}}\B_v$. Observe that if $R_{v}=0$, then there exists a path from $\mathbf{0}$ to $\mathbf{n}$ passing through $\B_{v}$ with length $\lfloor(2-\delta)n \rfloor$. We shall show that with high probability for many $i\in I$, there exists a positive fraction of $v$'s in $D_i$ with $R_v=0$. Call a vertex $v$ \textbf{slack} if $R_v\geq 1$. For $i\in I$, and $\e'>0$ call the anti-diagonal $D_i,$ as $\e'-$\textbf{saturated}  if the number of slack vertices on $D_i$ is at most $(1-\e') |D_i|$. Let $\e$ be as in Proposition \ref{FKGPoi}. Let $\cM_i$ denote the event that $D_i$ is $\e-$saturated. The following lemma is analogous to \eqref{e:manyb} and shows that there are many saturated anti-diagonals conditioned on $\cL_{\delta}$.

\begin{proposition}
\label{p:saturated}
In the above set-up, there exists $c>0$ such that for all $n$ sufficiently large
$$\P\left(\sum_{i\in I} \mathbf{1}(\cM_i)> \frac{\e m}{4} \middle| \cL_{\delta}\right) \geq 1-e^{-cn}$$
\end{proposition}

We postpone the proof of Proposition \ref{p:saturated} for now and complete the proof of Theorem \ref{t:delocLPP} first.

\begin{proof}[Proof of Theorem \ref{t:delocLPP}]
Fix $\e_0$ sufficiently small so that $\e_{D}|D_i|\e \ge 8\e_0n,$ for all $i\in I$. Fix $\gamma$ as in the statement of the theorem and consider $\gamma_n^{\e_0}$. Thus by choice of $\e_0$, for each $i\in I$, $$\#\{v\in D_i: \B_{v}\cap \gamma_{n}^{\e_0}\neq \emptyset\}<\frac{\e}{2} |D_i|.$$
Now for any $i\in I$ such that $\mathbf{1}(\cM_i)=1$ which means $D_i$ is $\e-$saturated,   it follows that there exists $v$ (in fact at least $\frac{\e}{2}|D_i|$ many $v's$) with $\B_v$ disjoint from $\gamma_{n}^{\e_0}$ such that $R_v=0$. This implies that there exists a path $\Gamma$ through $\B_v$ with length $\lfloor (2-
\delta)n \rfloor$. The proof is completed by observing that such a path must be a geodesic conditional on $\cL_{\delta}$.
\end{proof}

It remains to prove Proposition \ref{p:saturated} for which we need a number of lemmas. We start with the following lemma where for any point process $\Pi$, we denote by $|\Pi|$ the total number of points in the process. Recall the definition of $\Pi_v$ from  Section \ref{disc20}.
\begin{lemma}
\label{l:number}
For $v\in D_i$, the conditional law of $|\Pi_v|$ conditioned on $\cL_{\delta}$ and $\cF_{i}$ stochastically dominates a Poisson random variable with mean $\e_{D}^2$ conditioned to be at most $R_{v}$.
\end{lemma}
\begin{proof} Note that a standard way to sample $\Pi_v,$  is to first sample  $|\Pi_v|,$ which is a Poisson random variable with mean $\e_D^2$, and then choose $|\Pi_v|$ many uniformly located points in $\B_v.$ Now note that by definition of conditional probability, $\P(|\Pi_v|=k\mid \cF_i,\cL_{\delta})$ is proportional to $\P(|\Pi_v|=k)p_k$ where $p_k$ is  the probability ($\cF_i$ measurable), that the length of the longest path passing through $\B_v,$ where  $k$ points are uniformly chosen  in $\B_v,$ and the point process outside $D_i$ being specified by $\cF_i,$ is at most $(2-\delta)n.$ The proof of the lemma is now complete by noticing that $p_i=1$ for $i=0,1,2,\ldots,R_v.$
\end{proof}

Now we shall show that if almost every $v$ in any $D_i$ is slack, then the total number of points in $\cup_{v\in D_i} \B_v$ cannot be much smaller than typical, which will contradict Proposition \ref{FKGPoi}. For this we shall use the following lemma which says that if $\e_D$ is sufficiently small, conditioning the total number of points in $\B_{v}$ to be at most one, does not change the expected number of points by much. In the following lemma, $\Pi_{v}^{\leq 1}$ denotes a Poisson process on $\B_v$ conditioned to have at most one point,  and hence $|\Pi_{v}^{\leq 1}|$ is nothing but a Bernoulli variable.
\begin{lemma}
\label{l:meanPoi}
In the above set up, if $\e_D$ is sufficiently small (compared to $\e$), then we have for all $v$
$$\E|\Pi_v^{\le 1}|\ge (1-\e/4)\E|\Pi|.$$
\end{lemma}
\begin{proof}
Let $X$ be a Poisson variable with mean $\e_D^2$. The proof is completed by observing that $\E(X\mid X\le 1)=\frac{\e_D^2}{1+\e_D^2}$ and choosing $\e_D$ sufficiently small.
\end{proof}
Fix $\e_{D}$ so that the conclusion of Lemma \ref{l:meanPoi} holds. Let us define,
$$\cB_{i}=\left\{\sum_{v\in D_i} |\Pi_v|\le (1-2\e)|D_i|{\e^2_D}\right\}.$$ Observe that the proof of \eqref{e:manyb}, along with Proposition \ref{FKGPoi}, instead of Proposition \ref{FKG}, implies that,
\begin{equation}
\label{e:manybPoi}
\P\left(\sum_{i\in I} \mathbf{1}(\cB_i)> \frac{\e m}{4} \middle| \cL_{\delta}\right)\ge 1-e^{-cn}
\end{equation}
for some constant $c>0$ and for $n$ sufficiently large. 

The analogue of  Lemma \ref{l:bmest} relating $\cB_{i}$ to $\cM_{i}$ in this setting is the following.
\begin{lemma}
\label{l:bmestPoi}
There exists $c=c(\e)>0,$ such that for each $i\in I$ we have,
$$\P(\cB_i \cap \cM_i^c\mid \cL_{\delta}) \le e^{-cn}.$$
\end{lemma}

\begin{proof}
Observe that $\cM_i$ is $\cF_{i}$ measurable and it follows from Lemma \ref{l:distexcessLPP} and Lemma \ref{l:number} that (conditioned on $\cF_i$ and $\cL_{\delta}$), on $\cM_i^c,$ the total number of $\Pi$-points in $\cup_{v\in D_i} \B_v$ stochastically dominates the sum of $(1-\e)|D_i|$ many independent copies of $|\Pi_{v}^{\leq 1}|$. Note that the latter is nothing but a sum of $(1-\e)|D_i|$ many i.i.d. Bernoulli variables with mean $(1-\frac{\e}{4})\e^2_D$.
The result now follows from exponential concentration of such a sum.
\end{proof}

We can now finish the proof of Proposition \ref{p:saturated}.

\begin{proof}[Proof of Proposition \ref{p:saturated}]
By Lemma \ref{l:bmestPoi} and a union bound over all $i\in I$ it follows that
$$\P\left(\sum_{i\in I} \mathbf{1}(\cB_i\cap \cM_i^c)=0  \middle| \cL_{\delta}\right)\ge 1-e^{-cn}$$
for some $c>0$ and $n$ sufficiently large. This, together with \eqref{e:manybPoi} completes the proof of the proposition.
\end{proof}
Note that in the Poissonian setting we did not need analogues of Lemmas \ref{l:conditional} and \ref{l:goodb}. This is because  in the former case we were only interested in comparing the total number of saturated boxes and the total number of points on an anti-diagonal, and if the the number of saturated boxes were small, then the number of points was large with exponentially small failure probability which allowed the simple union bound argument used above to work. Whereas in the case of exponential passage times, in Lemma \ref{l:conditional} we compared the total weight on an anti-diagonal and occurrence of the strong concentration phenomenon, where the latter's failure probability is still a constant, precluding an union bound over the indices in $I$.

\section{Extensions to non-integrable settings}
\label{s:general}

So far we have proved our delocalization result for two models that are both exactly solvable (Poissonian LPP on $\R^2$ and Exponential LPP on $\Z^2$). Even though certain specific details of these models made the proofs simpler, they were not essential. Our argument, in its core, does not depend on exact solvability and in this section we provide examples of some general settings to which our results can be adapted.

\subsection{Last passage percolation on $\Z^2$ with general weights}
Consider last passage percolation on $\Z^2$ with general i.i.d.\ vertex weights $X_{v}$ coming from some distribution $F$ on the positive real line. As before let $L(u,v)$ denote the last passage time from $u$ to $v$ for $u\preceq v\in \Z^2$. (We shall make use of the same notations as before for other quantities as well). Under some fairly mild moment conditions on $F$, the law of large numbers result analogous to Theorem \ref{t:lln}, goes through. The following theorem was proved in \cite{M02} (see also \cite{CGGK93,GK94}).
\begin{theorem}
\label{t:lppgen}
Suppose $\int_{0}^{\infty} (1-F(x))^{1/2}~{\rm d}x <\infty$. Then there exists a function $G=G_{F}: \R^2_{+}\to \R_{+}$ such that for each $x,y>0$ we have
$$\lim_{n\to \infty} \frac{1}{n}\E L_{\lfloor nx \rfloor, \lfloor ny \rfloor}= G(x,y).$$
\end{theorem}

It is easy to see that such a $G$ is invariant under swapping of co-ordinates and is super additive, i.e.,  $G(x_1,y_1)+G(x_2,y_2)\leq G(x_1+x_2,y_1+y_2)$. The last observation implies that the set $\{(x,y):G(x,y)=1\}$ is the graph of a concave function, and it is believed that for a very general class of $F$, it is strictly concave. This model is also supposed to exhibit KPZ fluctuations of order $n^{1/3}$, for a general class of $F$, although it is known only the cases of Exponential, Geometric and Bernoulli distributed weights. Let us restrict ourself to the case where $F$ is continuous, so that almost surely there is a unique geodesic between any pair of points and as before  let $\Gamma_{n}$ denote the geodesic between $\mathbf{0}$ and $\mathbf{n}$. One difference from the exactly solvable cases is that it is not rigorously known that $\Gamma_{n}$ is concentrated around the diagonal line. An analogue of Theorem \ref{t:diag} can however be proved under the assumption of strict concavity of the limit shape and some nice tails of $F$, and hence is believed to be true for a general class of passage time distributions. Therefore it seems natural to consider the question of localization/ delocalization of geodesics in the lower tail large deviations regime.

As mentioned in Section \ref{pwoc}, the large deviation events are less well understood in the non-integrable setting, with no explicit formulae for the large deviation rate functions unlike the exactly solvable models. However it can be shown that the speed of the lower tail large deviations is of order $n^2$ as before. More precisely, let $F$ satisfy the hypothesis of Theorem \ref{t:lppgen} and let $\mu=G(1,1)$. Fix $\delta\in (0,\mu)$. It can be shown, following the argument outlined in Section \ref{s:outline} (using exponential concentration below the mean for sums of i.i.d. positive random variables), that
$$ -\infty < \liminf_{n\to \infty} \frac{\log \P (L_n \leq (\mu-\delta)n)}{n^2} \leq  \limsup_{n\to \infty} \frac{\log \P (L_n \leq (\mu-\delta)n)}{n^2} < 0,$$
(also see \cite{Kesten}).

We shall show, that under certain additional assumptions, the analogue of Theorem \ref{t:deloc} remains valid in this setting. Before making a formal statement we shall need to define two classes of probability measures.

\begin{definition}
\label{d:prob}
Let $\cP$ denote the class of all probability measures with support $[0,\infty),$ with continuous and positive density that satisfy the hypothesis of Theorem \ref{t:lppgen}. Let $\cP_1\subseteq \cP$ denote the class of probability measures with non-increasing density and let $\cP_2\subseteq \cP$ denote the class of all probability measures with log concave density, i.e., density of the form $e^{-V(\cdot)}$
where $V(\cdot)$ is a convex function that is continuously differentiable on $[0,\infty)$ with $V'(0)> -\infty$.
\end{definition}

Our main result in this section is to show that the delocalization result Theorem \ref{t:deloc} remains valid in the setting of last passage percolation with general i.i.d.\ weights as long as the weights come from a distribution in $\cP_1$ or $\cP_2$. Recall the notion of an $\e$-cylinder $\gamma_{n}^{\e}$ around a continuous surjective increasing function $\gamma: [0,1]\to [0,1]$ (see \eqref{sausage100}).

\begin{maintheorem}
\label{t:delocgen}
Let $F$ be a probability measure that is either in $\cP_1$ or in $\cP_2$. Let $\mu=G_{F}(1,1)$ where $G_{F}$ is as in Theorem \ref{t:lppgen}. Fix $\delta\in (0,\mu)$ and $\e>0$, and set $\cL_{\delta}:=\{L_n \leq (\mu-\delta)n\}$. There exists $\e'>0$ such that for all $\gamma: [0,1]\to [0,1]$ surjective and increasing one has
$$\P(\Gamma_{n}\subseteq \gamma_n^{\e'}\mid \cL_{\delta})\leq \e$$
for all $n\in \N$.
\end{maintheorem}

It will be clear from the proof that the condition that $F$ is in $\cP_1$ or in $\cP_2$ is not optimal, even for our argument. We have not attempted to find the most general class of distributions for which our proof works. With a view of keeping the exposition as simple as possible, our objective was to find a class of distributions that is sufficiently general to be of interest. Observe also that the special case of exponential distribution treated in Theorem \ref{t:deloc} falls into both the classes $\cP_1$ and $\cP_2$.

The proof of this theorem follows along the same lines as that of Theorem \ref{t:deloc}. That is, we first show that conditional  on $\cL_{\delta},$  the length of the geodesic is concentrated at scale $\frac{1}{n}$. As already mentioned in Remark \ref{gen304}, this part of the argument does not use any crucial property of the Exponential distribution and will go through for any $F$ in either $\cP_1$ or $\cP_2$. The anti-concentration part however requires more work since we can no longer exploit nice decoupling properties as in Fact \ref{fact1}. We shall prove the following proposition which is
Theorem \ref{res2} restated in this setting. Recall that for $A\subseteq \llbracket 0,n \rrbracket^2$, $L_n(A)$ denotes the length of the maximal path between $\mathbf{0}$ and $\mathbf{n}$ among all paths completely contained in $A$.

\begin{proposition}
\label{p:resAgeneral}
Let $F$ be a probability distribution in $\cP_1$ or $\cP_2$, and consider the set-up in Theorem \ref{t:delocgen}. Let $H$ and $\e_3>0$ be fixed. Then for $\e'>0$ sufficiently small and for all $A\subseteq \llbracket 0,n \rrbracket^2$ with $|A|\leq \e' n^2$ we have
$$\P\left(L_n(A) \ge (\mu-\delta)n -\frac{H}{n}\mid \cL_{\delta}\right)\le \e_3.$$
\end{proposition}

As in the application of Fact \ref{fact1} in Lemma \ref{l:sample}, fixing $n,$ we study the field $\mathbf{X}=(X_v: v \in \llbracket 0,n \rrbracket^2)=(\mathbf{X}_A,\mathbf{X}_{A^c})$ through the
triple $(\mathbf{X}_{A^c},\mathbf{Y}_{A}, Z_{A})$ where all the elements were defined during the proof of Lemma \ref{l:sample}.

By hypothesis $(X_v: v \in \llbracket 0,n \rrbracket^2)$ are distributed as i.i.d. random variables following a distribution $F \in \cP_{1} \cup \cP_2$, say with mean $m$ and density, $f(\cdot)$.
A  simple change of variable shows that conditioning on $(\mathbf{X}_{A^c},\mathbf{Y}_{A})=(\mathbf{x}_{A^c},\mathbf{y}_{A}),$  the density of $Z_A$ at any $z>0$ is proportional to
\begin{equation}\label{cond243}
z^{|A|-1}\prod_{v\in A}f(zy_v).
\end{equation}
Note that given $\mathbf{X}_{A^c}, \mathbf{Y}_{A},$ the quantities $L_n$ and $L_n(A)$, are non-decreasing and  strictly increasing functions of $Z_A$ respectively. As before
we will call $(\mathbf{x}_{A^c}, \mathbf{y}_{A})$ as compatible with $\cL_{\delta}$, if there exists some $z_{A}>0,$ such that $(\mathbf{x}_{A^c}, \mathbf{y}_{A},z_A)\in \cL_{\delta},$ (note that this in fact is just a property of $\mathbf{x}_{A^c}.$)

Now for any $\mathbf{y}_{A}\in \cS_{|A|}$ (recall that $\cS_{|A|}$ is the simplex of dimension $|A|-1$) and compatible $\mathbf{x}_{A^c}$ define
\begin{equation}\label{def203}
\theta_{\max}:=\theta_{\max}(\mathbf{y}_{A}, \mathbf{x}_{A^c}):=\min\left(\sup\{\theta:(\theta\mathbf{y}_{A}, \mathbf{x}_{A^c})\in \cL_{\delta}\},2m|A|\right),
\end{equation}
where $\theta\mathbf{y}_{A}=(\theta y_{v}: v \in A)$ and $m$ is the mean of the distribution $F$. We now have the following lemma analogous to Lemma \ref{l:sample}.
\begin{lemma}
\label{l:samplegeneral}
For any $\mathbf{y}_{A}\in \cS_A$ and compatible $\mathbf{x}_{A^c},$
 conditional on $(\mathbf{Y}_{A}=\mathbf{y}_{A}, \mathbf{X}_{A^c}=\mathbf{x}_{A^c})$ and the events $\{Z_A\le 2m |A|\}$, and $\cL_\delta$: the distribution of $Z_{A}$ is supported on $[0,\theta_{\max}]$ and has the following density at $z\in (0,\theta_{\max})$:
$$\frac{z^{|A|-1}\prod_{i=1}^{|A|}f(zy_v)}{\int_{0}^{\theta_{\max}}w^{|A|-1}\prod_{i=1}^{|A|} f(wy_v){\rm{d}}w}.$$
\end{lemma}
\begin{proof}
The proof is a straightforward consequence of \eqref{cond243} and the definition of $\theta_{\max}.$
\end{proof}

Before the proof of Proposition \ref{p:resAgeneral} we need another short lemma. Let us abbreviate the event $\{Z_A\le 2m |A|\}$ by $\ce_{A}$.

\begin{lemma}
\label{prep23}
For any $\by_A\in \cS_A$ and compatible $\bx_{A^c},$
\begin{align}\label{exp23}
\P\left(L_n(A) \ge (\mu-\delta)n -\frac{H}{n}\middle| \by_A,  \bx_{A^c}, \ce_{A}, \cL_{\delta} \right) \le \P\left(Z_{A}\ge \theta_{\max}(1-\frac{M}{n^2})\middle | \by_A, \bx_{A^c}, \ce_{A}, \cL_{\delta}\right)
\end{align}
where $M=\frac{H}{\mu-\delta}.$
\end{lemma}
\begin{proof} Recalling the notation $L_n(A; \bY_{A})$ from Lemma \ref{l:sample}, clearly, $L_n(A)=L_n(A;\bX_{A})=Z_AL_n(A;\bY_{A}).$
Thus by definition,
\begin{equation}\label{ub234}
\theta_{\max}L_n(A;\bY_{A})\le (\mu-\delta)n.
\end{equation}
Hence  \begin{align*}
L_n(A;\bX_{A})&\ge (\mu-\delta)n -\frac{H}{n}
\implies Z_AL_n(A;\bY_{A})\ge (\mu-\delta)n -\frac{H}{n},\\
&\overset{\eqref{ub234}}{\implies}\frac{Z_A}{\theta_{\max}}\ge \frac{(\mu-\delta)n -\frac{H}{n}}{(\mu-\delta)n}.
\end{align*}
\end{proof}

We are now ready to prove Proposition \ref{p:resAgeneral}.
\begin{proof}[Proof of Proposition \ref{p:resAgeneral}] Let $s=1-\frac{M}{n^2}$ where $M$ is defined in the statement of Lemma \ref{prep23}.
It follows that
\begin{eqnarray}
\label{e:resA12}
\P\left(L_n(A) \ge (\mu-\delta)n -\frac{H}{n}\middle| \cL_{\delta}\right)&\le
\P\left(\ce_{A}\middle| \cL_{\delta}\right)\P\left(L_n(A) \ge (\mu-\delta)n -\frac{H}{n}\middle| \ce_{A},\cL_{\delta}\right)+\P(\ce_{A}^{c}\mid \cL_{\delta})\\ \nonumber
&\overset{\eqref{exp23}}{\le}  \P(\ce_{A}\mid \cL_{\delta})\E \biggl[ \P(Z_{A}\geq s\theta_{\max}\mid \mathbf{y}_{A}, \bx_{A^c}, \ce_{A}, \cL_{\delta}) \biggr]+\P(\ce_{A}^{c}\mid \cL_{\delta}),
\end{eqnarray}
where the expectation is over the distribution of $(\mathbf{y}_{A}, \bx_{A^c})$ conditional on the events $\ce_{A}$ and $\cL_{\delta}$. Using Lemma \ref{l:samplegeneral}, it follows that
\begin{align*}
\P(Z_{A}\geq s\theta_{\max}\mid (\mathbf{y}_{A}, \bx_{A^c}), \ce_{A}, \cL_{\delta}) &=\dfrac{\int_{s\theta_{\max}}^{\theta_{\max}}z^{|A|-1}\prod_{v\in A}f(zy_v){\rm{d}}z}{\int_{0}^{\theta_{\max}}z^{|A|-1}\prod_{v\in A} f(zy_v){\rm{d}}z}.
\end{align*}
Doing the change of variable $\frac{z}{\theta_{\max}}\mapsto t,$ we get,
\begin{align}
\nonumber
\dfrac{\int_{s\theta_{\max}}^{\theta_{\max}}z^{|A|-1}\prod_{v\in A}f(zy_v){\rm{d}}z}{\int_{0}^{\theta_{\max}}z^{|A|-1}\prod_{v\in A} f(zy_v){\rm{d}}z}&=
\dfrac{\int_{s}^{1}t^{|A|-1}\prod_{v\in A} f(t\theta_{\max}y_v){\rm{d}}t}{\int_{0}^{1}t^{|A|-1}\prod_{v\in A} f(t\theta_{\max}y_v){\rm{d}}t}\\
\label{requiredbnd12}
&\le \dfrac{\int_{s}^{1}t^{|A|-1}\prod_{v\in A} f(t\theta_{\max}y_v){\rm{d}}t}{\int_{1-\frac{1}{|A|}}^{1}t^{|A|-1}\prod_{v\in A} f(t\theta_{\max}y_v){\rm{d}}t}.
\end{align}
To complete the proof, we will show that $\dfrac{\int_{s}^{1}t^{|A|-1}\prod_{v\in A} f(t\theta_{\max}y_v){\rm{d}}t}{\int_{1-\frac{1}{|A|}}^{1}t^{|A|-1}\prod_{v\in A} f(t\theta_{\max}y_v){\rm{d}}t}$ is small, when $|A|=\e' n^2$ for some small $\e'$.
Note that we can ignore the term $t^{|A|-1},$ in the numerator and denominator since it is $\Theta(1)$ when $t\in [1-\frac{1}{|A|},1].$
Now note that,
\begin{align}\label{decompose}
\int_{1-\frac{1}{|A|}}^{1}\prod_{v\in A} f(t\theta_{\max}y_v){\rm{d}}t=\sum_{i=0}^{\frac{n^2}{M|A|}-1}\int_{1-(i+1)\frac{M}{n^2}}^{1-i\frac{M}{n^2}}\prod_{v\in A} f(t\theta_{\max}y_v){\rm{d}}t,
\end{align} where to avoid rounding issues we assume $\frac{n^2}{M|A|}$ is an integer.
If $F\in \cP_1,$ then just using monotonicity $i\in\{0,1,\ldots,\frac{n^2}{M|A|}-1\}$
 we have \begin{align*}
\int_{1-(i+1)\frac{M}{n^2}}^{1-i\frac{M}{n^2}}\prod_{v\in A} f(t\theta_{\max}y_v){\rm{d}}t\ge\int_{1-\frac{M}{n^2}}^{1}\prod_{v\in A} f(t\theta_{\max}y_v){\rm{d}}t
\end{align*}and hence the required bound on the RHS in \eqref{requiredbnd12} follows.
However to prove a similar bound when  $F\in \cP_2,$ note that for $i\in\{0,1,\ldots,\frac{n^2}{M|A|}-1\},$
\begin{align*}
\int_{1-(i+1)\frac{M}{n^2}}^{1-i\frac{M}{n^2}}\prod_{v\in A} f(t\theta_{\max}y_v){\rm{d}}t&\ge \left[\inf_{t\in [1-(i+1)\frac{M}{n^2},1-i\frac{M}{n^2}]}\frac{\prod_{v\in A} f(t\theta_{\max}y_v)}{\prod_{v\in A}f((t+\frac{iM}{n^2})\theta_{\max}y_v)}\right]\int_{1-\frac{M}{n^2}}^{1}\prod_{v\in A} f(t\theta_{\max}y_v){\rm{d}}t,\\
&\ge C\int_{1-\frac{M}{n^2}}^{1}\prod_{v\in A} f(t\theta_{\max}y_v){\rm{d}}t.
\end{align*}
for some $C>0$ (not depending on $|A|$)
where the last inequality follows from Lemma \ref{comparison} below and that  $\theta_{\max}\le 2m|A|.$
Hence whenever $F\in \cP_1\cup \cP_2,$ using the above bound and \eqref{decompose},
\begin{align*}
\dfrac{\int_{{1-\frac{M}{n^2}}}^{1}t^{|A|-1}\prod_{v\in A} f(t\theta_{\max}y_v){\rm{d}}t}{\int_{1-\frac{1}{|A|}}^{1}t^{|A|-1}\prod_{v\in A} f(t\theta_{\max}y_v){\rm{d}}t}\le \frac{1}{C\frac{n^2}{M|A|}}
= O(\frac{|A|}{n^2}).
\end{align*}
Plugging the above in \eqref{e:resA12},  along with the fact that $\P(\ce_{A}^{c}\mid \cL_{\delta})$ goes to $0$ completes the proof. To see that $\P(\ce_{A}^c \mid \cL_{\delta})$ goes to zero, note that by the FKG inequality $$\P(Z_A\ge 2m|A|\mid \cL_{\delta})\le \P(Z_A\ge 2m|A|)$$ and the latter goes to zero by law of large numbers as $\E(Z_A)=m|A|.$
\end{proof}

\begin{lemma}\label{comparison} For any density function $f$ corresponding to a probability measure in $\cP_2,$ then there exists $C>0$ such that for any $\bx_A$, and any $0<t<1,$
\begin{align*}
\frac{\prod_{v\in A} f(x_v)}{\prod_{v\in A} f(tx_v)}<e^{C(1-t)z_A}.
\end{align*}
\end{lemma}
\begin{proof}Since $\log f(\cdot)=-V(\cdot),$ it suffices to show that
$$\sum_{v\in A}V(tx_v)-\sum_{v\in A}V(x_v)<C(1-t)\sum_{v\in A}x_v.$$
Now  note that by hypothesis $\inf_{x\in \R_+}V'(x)=-C>-\infty.$  The proof now is a straightforward consequence of mean value theorem.
\end{proof}

We can now complete the proof of Theorem \ref{t:delocgen}.

\begin{proof}[Proof of Theorem \ref{t:delocgen}]
Fix a probability distribution $F$ either in $\cP_1$ or in $\cP_2$, $\delta\in (0,\mu)$ and $\e>0$. Fix also an increasing surjective function $\gamma: [0,1]\to [0,1]$. Arguing verbatim as in the proof of Theorem \ref{res1} by Remark \ref{gen304},  it follows that there exists $H$ such that
$$\P\left(L_n\geq (\mu-\delta)n -\frac{H}{n}\mid \cL_{\delta}\right)\geq 1-\e/2.$$
Now by Proposition \ref{p:resAgeneral}, one can choose $\e'$ sufficiently small so that  $|\gamma_n^{\e'}|\leq 2\e' n^2$ and
$$\P\left(L_n(\gamma_n^{\e'})\geq (\mu-\delta)n -\frac{H}{n}\mid \cL_{\delta}\right)\leq \e/2.$$
Combining the above we get $\P(L_n\neq L_n(\gamma_n^{\e'})\mid \cL_{\delta})\ge 1-\e.$
\end{proof}

\subsection{Last passage percolation on $\Z^d$}\label{highdim}
Another setting to which our argument extends in a rather straightforward way is that of directed last passage percolation on higher dimensional Euclidean lattices $\Z^d$ with i.i.d. weights on the vertices. The last passage percolation model can be defined on $\Z^d$ for $d>2$, by extending the definition on $\Z^2$ in an obvious way. Given positive i.i.d.\ weights $\{X_v:v\in \Z^d\}$ one defines the last passage time $L_n$ from $\mathbf{0}=(0,\ldots, 0)$ to $\mathbf{n}=(n,\ldots , n)$ maximizing the weight over all paths that are co-ordinate wise non-decreasing and the weight of a path as before,  is the sum of the weights of the vertices on it. Let $\Gamma_n$ denote the maximizing path (which is unique if $F$ is continuous which will be the case we shall be restricted to).

The law of large number result (i.e., the analogue of Theorem \ref{t:lppgen}) holds in this case provided $\int_{0}^{\infty} (1-F(x))^{1/d}~{\rm d}x < \infty$ (see \cite{M02}). Let us define classes of probability distributions $\cP_1^{(d)}$ and $\cP_2^{(d)}$ exactly as in Definition \ref{d:prob} except that we now also require the above tail condition. Let $\mu=\mu_{F}:=\lim_{n\to \infty} \frac{\E L_n}{n}$. For any $\delta\in (0,\mu)$ the following is known about the lower tail large deviation event $L_n \leq (\mu-\delta)n$
$$ -\infty < \liminf_{n\to \infty} \frac{\log \P (L_n \leq (\mu-\delta)n)}{n^d} \leq  \limsup_{n\to \infty} \frac{\log \P (L_n \leq (\mu-\delta)n)}{n^d} < 0.$$
This can be proved by an easy adaptation of the argument in \cite{Kes86} for first passage percolation. As before let $\cL_{\delta}:=\{L_n\leq (\mu-\delta)n\}$ denote this large deviation event.

Finally for a fixed continuous increasing function $\gamma: [0,1]\to [0,1]^d$ such that $\gamma(0)=\mathbf{0}$ and $\gamma(1)=\mathbf{1},$  one can define the $\e$-cylinder around $\gamma$ to be the set of all points at distance at most $\e$ from the image of $\gamma,$ i.e., $\gamma([0,1])$
and denote by $\gamma_{n}^{\e}$, the image of the $\e-$cylinder under the scaling map  $\bx \mapsto n\bx$. The following is our delocalization result in the higher dimensional setting whose proof is identical to that of Theorem \ref{t:delocgen} and hence shall be omitted (Note that the strong concentration for $L_n$ analogous to Theorem \ref{res1}  now occurs at scale $\frac{1}{n^{d-1}}.$).

\begin{maintheorem}
\label{t:delochd}
Let $F$ be a probability measure that is either in $\cP_1^{(d)}$ or in $\cP_2^{(d)}$. Consider directed last passage percolation on $\Z^d$ ($d>2$) and let $\mu$ be as above. Fix $\delta\in (0,\mu)$ and $\e>0$. There exists $\e'>0$ such that for all $\gamma$ as above:  one has
$$\P(\Gamma_{n}\subseteq \gamma_n^{\e'}\mid \cL_{\delta})\leq \e$$
for all $n\in \N$.
\end{maintheorem}

\bibliography{delocalization}
\bibliographystyle{plain}
\end{document}